\documentclass[12pt, a4paper]{amsart}

\usepackage{amsmath, amsthm, amssymb}
\usepackage{fancyhdr}
\usepackage{graphicx}
\usepackage{enumerate}
\usepackage[top=2.5cm, bottom=2.5cm, left=2.5cm, right=2.5cm]{geometry}

\newtheorem{definition}{Definition}
\newtheorem{lemma}{\bf Lemma}
\newtheorem{proposition}{\bf Proposition}
\newtheorem{theorem}{\bf Theorem}
\newtheorem{corollary}{\bf Corollary}

\renewenvironment{proof}{\noindent {\bf Proof: }}{\rm\\}
\theoremstyle{definition}
\newtheorem{remark}{Remark}{\rm}
\newtheorem{example}[subsection]{Example}{\rm}

\usepackage{color}
\usepackage[T1]{fontenc}
\usepackage{mathrsfs}
\usepackage{soul}

\usepackage{auto-pst-pdf}
\usepackage{pst-plot}
\usepackage{pdftricks}
\begin{psinputs}
\usepackage{pstricks,pst-plot}
\end{psinputs}

\newcommand{\R}{\mathbb{R}}
\newcommand{\ball}{\mathcal{B}}

\begin{document}

\title{On local convergence
of the method of alternating projections}
\author{Dominikus Noll$^*$, Aude Rondepierre$^\dag$}
\thanks{$^*$Universit\'e de Toulouse, Institut de Math\'ematiques, 118, route de Narbonne, 31062 Toulouse, France,
Tel. 0033 5 61 55 86 22, Fax. 0033 5 61 55 83 85. E-mail  {\tt noll@math.univ-toulouse.fr}. Corresponding author.}
\thanks{$^\dag$Universit\'e de Toulouse, Institut de Math\'ematiques, 118 route de Narbonne, 31062 Toulouse, France. Tel. 0033 5 61 55 76 48. E-mail: {\tt aude.rondepierre@insa-toulouse.fr}}
\date{August 11, 2014}

\maketitle

\begin{abstract}
The method of alternating projections is a classical tool to solve feasibility
problems.
Here we prove local convergence of alternating projections between subanalytic sets $A,B$
under a mild regularity hypothesis on one of the sets.  We show that the speed of convergence is 
$\mathcal O(k^{-\rho})$ for some $\rho\in (0,\infty)$.

\vspace*{.1cm}

\vspace*{.1cm}
\end{abstract}

\section{Introduction}
The method of alternating projections is a classical tool to
solve the following feasibility problem: 
Given closed sets $A,B$ in $\mathbb R^n$, find a point $x^*\in A \cap B$.
Alternating projections can be traced back to the work of  Schwarz \cite{schwarz} in 1869, and were
popularized in lecture notes of von Neumann \cite{neumann} since the 1930s.
The method generates sequences $a_k\in P_A(b_{k-1})$,
$b_k\in P_B(a_k)$, where $P_A$, $P_B$ are the set-valued orthogonal projection
operators on $A$ and $B$. If the alternating sequence $a_k,b_k$ is bounded and satisfies
$a_k-b_k\to 0$, then each of its accumulation points
is a solution of the feasibility problem. 
The fundamental question is
when  such a sequence converges to a single limit point $x^*\in A \cap B$.
 
For convex sets alternating projections are globally convergent as soon as $A \cap B \not=\emptyset$, and
the survey \cite{bauschke-survey} gives an excellent state-of-art of the convex theory. 
In one of the earliest contributions to the nonconvex case, Combettes and Trussell \cite{combettes}
proved in 1990 that the set  of accumulation points
of a bounded sequence of alternating projections with $a_k-b_k\to 0$ is either a singleton
or a nontrivial compact
continuum. In 2013 it was shown in \cite{bauschke-noll}  by way of an example that 
the continuum case may indeed occur. This shows that without convexity a sequence of alternating projections
$a_k,b_k$ may fail to converge even when it is bounded and satisfies  $a_k-b_k\to 0$.

In 2008 Lewis and Malick \cite{malick} proved that a sequence $a_k,b_k$ of alternating
projections converges locally linearly
if $A,B$ are $C^2$-manifolds intersecting transversally. Expanding on this in 2009,  Lewis {\em et al.}
\cite{luke}  proved local linear convergence for general $A,B$ intersecting non-tangentially in the sense of
linear regularity,
where one of the sets is superregular. In 2013 Bauschke {\em et al.} {\cite{bauschke1,bauschke2}} investigate the case of non-tangential intersection further and prove linear convergence under 
weaker regularity and transversality hypotheses.

Here we prove local convergence  under less restrictive conditions, where  $A,B$ may
also intersect tangentially. 
We propose a new geometric concept, called {\em separable intersection}, which 
gives local convergence of alternating projections   when combined with
{\em H\"older regularity}, a mild 
hypothesis less restrictive than prox-regularity. 

Separable intersection has wide scope for applications, as it not only
includes non-tangential intersection, 
but goes beyond and allows also   a large variety of
cases where $A, B$ intersect   tangentially. In particular, we prove 
that closed subanalytic  sets $A,B$ {\em always}
intersect separably. This leads to the
central result that  alternating projections
between subanalytic sets converge locally with rate $\mathcal O(k^{-\rho})$ for some
$\rho\in (0,\infty)$  if one of the sets is H\"older regular with respect to the other. 
As these hypotheses are satisfied
in practical situations, we obtain a theoretical explanation
for the fact, observed in practice,  that even without convexity alternating projections converge
well in the neighborhood of $A\cap B$. As an application, we obtain a
local convergence proof for the classical Gerchberg-Saxton error reduction algorithm in phase retrieval.

The structure of the paper is as follows. Section \ref{tangential}
introduces the concept of separable intersection of two closed sets. 
Then $0$-separability is related to existing transversality concepts.
In section \ref{hoelder} we discuss
H\"older regularity and compare it to older
regularity concepts like prox-regularity, Clarke regularity, and superregularity.
The central chapter \ref{convergence} gives the convergence proof with rate
for sets intersecting separably. In section \ref{loja} we show that subanalytic sets intersect
separably and then deduce the convergence result for subanalytic sets.
Section \ref{loja} gives also some applications indicating the versatility of our convergence test.
In particular, we prove local convergence of  an averaged projection method related  to
in \cite[Corollary 12]{attouch}, where the authors use the Kurdyka-\L ojasiewicz inequality.
The final section \ref{examples} gives limiting examples. 

After the initial version \cite{initial} of this article, a concept related to our notion of $0$-separability,
called intrinsic transversality, was announced in \cite{obsolete}. 
We compare this to our own transversality and regularity concepts
in sections \ref{tangential} and \ref{hoelder}.

\section{Preparation}
Given a nonempty closed subset $A$ of $\mathbb R^n$, the projection
onto $A$ is the set-valued mapping $P_A$
associating with $x\in \mathbb R^n$ the nonempty set
\[
P_A(x)=\left\{a\in A: \|x-a\|=d_A(x)  \right\},
\]
where $\|\cdot\|$ is the Euclidean norm,  induced by the scalar product $\langle \cdot,\cdot\rangle$, 
and where
$d_A(x)=\min \{\|x-a\|: a\in A\}$.   The closed Euclidean ball with center $x$ and radius $r$ is
denoted $\ball(x,r)$.
We write $a\in P_A(b)$ if the projection is 
potentially set-valued, while $a=P_A(b)$ means it is unique.

A sequence of alternating projections between nonempty closed sets $A,B$ 
satisfies $b_k\in P_B(a_k), a_{k+1}\in P_A(b_k)$, $k\in \mathbb N$. 
We occasionally switch to the following index-free notation, which is standard in optimization: 
$$b\in P_B(a), a^+\in P_A(b), b^+\in P_B(a^+), {\rm etc.}$$ 
The sequence of alternating projections
is then $\dots, a,b,a^+,b^+,a^{++},b^{++},\dots$.
We refer to $a \to b \to a^+$, respectively $b\to a^+\to b^+$, 
as the building blocks of the sequence, where  it is always
understood that 
$b\in P_B(a)$, $a^+\in P_A(b)$, $b^+\in P_B(a^+)$, etc. 

Notions from nonsmooth analysis are covered by \cite{rock,mord}.  The proximal normal
cone to $A$ at $a\in A$ is the set
$N_A^p(a)=\{\lambda u: \lambda \geq 0, x \in P_A(x+u)\}$.
The normal cone to $A$ at $a\in A$ is the set 
$N_A(a)$ of vectors $v$ for which there exist $a_k\in A$ with $a_k \to a$ and $v_k\in N_A^p(a_k)$
such that $v_k\to v$. The Fr\'echet normal cone $\widehat{N}_A(a)$ to $A$ at $a\in A$
is the set of $v$ for which $\limsup_{A \ni a'\to a} \frac{\langle v,a'-a\rangle}{\|a'-a\|}\leq 0$; cf. \cite[(1.2)]{mord}.
We have the inclusions $N_A^p(a)\subset \widehat{N}_A(a)\subset N_A(a)$; cf. {\cite[Chapter 2.D and (1.6)]{mord} or \cite[Lemma 2.4]{bauschke1}}.
{For any function $f:\mathbb{R}^n\rightarrow \mathbb{R}\cup\{\infty\}$, the epigraph of $f$ is the set 
${\rm epi}f=\{(x,\xi)\in\mathbb{R}^n\times \mathbb{R} : \xi \geq f(x)\}$.}
The proximal subdifferential $\partial_pf(x)$ of a lower semi-continuous
function $f$ at $x\in {\rm dom}f$ is the set of vectors 
$v\in \mathbb R^n$ such that $(v,-1)\in N^p_{{\rm epi}f}(x,f(x))$; \cite[(2.81)]{mord}. The 
subdifferential $\partial f(x)$ of $f$ at $x\in {\rm dom} f$ is the set of $v$
satisfying $(v,-1)\in N_{{\rm epi} f}(x,f(x))$. 
The Fr\'echet subdifferential 
$\widehat{\partial}f(x)$ at $x\in {\rm dom} f$ is the set of $v\in \mathbb R^n$ such that
$(v,-1)\in \widehat{N}_{{\rm epi}f}(x,f(x))$, cf.  \cite[(1.51)]{mord}.

\section{Tangential and non-tangential intersection}
\label{tangential}

In this section we introduce the fundamental  concept
of separable intersection of sets $A,B$, which plays the central role
in our convergence theory. 

\begin{definition}
\label{separable}
{\bf (Separable intersection)}.
We say that $B$ intersects $A$  separably at $x^*\in A \cap B$ with exponent $\omega \in [0,2)$
and constant $\gamma>0$  if there exists a neighborhood
$U$ of $x^*$ such that for every building block $b\to a^+\to b^+$ in $U$,  the condition
\begin{eqnarray}
\label{angle}
\langle b-a^+,b^+-a^+\rangle 
\leq
\left( 1-\gamma \|b^+-a^+\|^\omega \right) \|b-a^+\| \|b^+-a^+\|
\end{eqnarray}
is satisfied.
\hfill $\square$
\end{definition}
We say that 
$B$ intersects $A$ separably at $x^*$ if (\ref{angle}) holds for {\em some} $\omega\in [0,2)$, $\gamma > 0$. 
If it is also true that $A$ intersects $B$ separably, that is, if the analogue of (\ref{angle}) holds 
for building blocks $a \to b \to a^+$, then we obtain a symmetric condition,
and in that case we  say that $A,B$ intersect separably at $x^*$.

\begin{remark}
Condition (\ref{angle}) discloses itself if we introduce the
angle $\alpha = \angle(b-a^+,b^+-a^+)$ and rewrite
(\ref{angle}) in the more suggestive form

\begin{equation*}
   \makebox[\linewidth]{($\ref{angle}^\prime$) \hspace{5.4cm}$\displaystyle\frac{1-\cos\alpha}{\|a^+-b^+\|^\omega} \geq \gamma$, \hfill}
\end{equation*}
calling this the
{\em angle condition} for the building block $b\to a^+\to b^+$.  For $\omega \in (0,2)$
the interpretation of (\ref{angle}), or $(1')$,  is that if the angle $\alpha$  between 
$b-a^+$ and $b^+-a^+$ for two consecutive projection steps
$b\to a^+ \to b^+$
shrinks down to $0$ as the alternating sequence approaches $x^*$, then  $\alpha$ should not shrink too fast.
Namely, through $(1')$,  
the angle is linked to the shrinking distance between the sets. For $\omega=0$
the meaning of $(1')$ is that the angle $\alpha$ stays away from $0$.
\end{remark}

\begin{remark}
\label{incr}
Suppose $B$ intersects $A$ separably with exponent $\omega\in [0,2)$ and constant $\gamma>0$ at $x^*$. 
Let $\omega'\in (\omega,2)$ and $\gamma ' \in( 0,\gamma]$. Then $B$ intersects $A$ also $\omega'$-separably
with constant $\gamma'$. 
In consequence,
$0$-separability is the severest condition, while $\omega$-separability gets less restrictive as $\omega$ increases.

As we shall see, for
$\omega \geq 2$ property ($\ref{angle}$) can still be formulated, but turns out too weak
to be meaningful. For an  illustration see example \ref{weak}.
\end{remark}

\begin{remark}
Informally, when the  angle $\alpha = \angle(b-a^+,b^+-a^+)$ between two
consecutive projection steps  shrinks to zero, 
$A,B$ must in some sense  intersect {\em tangentially} at $x^*$.
In contrast, when $\alpha$ stays away from 0, the case of $0$-separability,
one could say that $A,B$ intersect transversally, or  {\em at an angle}. In that  case
alternating projections  are expected to behave well and converge linearly.  Tangential
intersection is the more embarrassing case, where convergence could 
be  slowed down or even fail. Our concept of
$\omega$-separability gives  new insight into the case of tangential intersection.
\end{remark}

There has been
considerable effort in the literature  to {\em avoid} tangential intersection
by making transversality assumptions. We mention
transversal intersection in \cite{malick}, the generalized non-separation property
in \cite{mord}, linearly regular intersection in \cite{luke},
or the notion of constraint qualification in 
\cite{bauschke1}. In the following we relate these notions to 
$0$-separability.

Bauschke {\em et al.}  \cite[Definition 2.1]{bauschke1} introduce an extension of the Mordukhovich
normal cone called the $B$-restricted  normal cone $N_A^B(x^*)$ to $A$ at $x^*\in A$.  They define $u\in N_A^B(x^*)$
if there exist $a_n\in A$, $a_n\to x^*$, and $u_n \to u$ such that
\[
u_n = \lambda_n \left(  b_n-a_n \right)
\]
for some $\lambda_n > 0$ and $b_n\in B$ with $a_n\in P_A(b_n)$. They then establish basic inclusions between the restricted normal cone and various classical cones \cite[Lemma 2.4]{bauschke1}. In particular for any $a\in A$ and $B$  one has
$
N_A^B(a)\subset N_A(a)
$.

Now 
let $\widetilde A$ and $\widetilde B$ be non-empty subsets of $\mathbb{R}^n$. In  \cite[Definition 6.6]{bauschke1}
the authors say that  $(A,\widetilde A,B,\widetilde B)$ satisfies the CQ-condition at $x^*\in A\cap B$ if
\begin{equation}
N_A^{\widetilde B}(x^*) \cap \left(  -N_B^{\widetilde A}(x^*) \right) \subset \{0\}.\label{CQ:condition}
\end{equation}
This condition is to be understood as a transversality hypothesis, because we have the following

\begin{proposition}
\label{prop1}
{\bf (CQ  implies  $0$-separability)}. 
Let $P_A(\partial B\setminus A) \subset \widetilde{A},~P_B(\partial A\setminus B)\subset \widetilde{B}$, and
suppose $(A,\widetilde A, B, \widetilde B)$ satisfies the CQ-condition at $x^*\in A \cap B$.
Then $A,B$ intersect $0$-separably at $x^*$.
\end{proposition}

\begin{proof}
According to \cite[Definition 2.1]{bauschke1} the $\widetilde{B}$-restricted proximal normal cone $\widehat{N}_A^{\tilde B}(a)$ of $A$
at $a\in A$ is the set of vectors $u$ of the form $u =  \lambda(\tilde{b}-a)$ for some $\lambda > 0$ and some
$\widetilde{b}\in \widetilde{B}$ satisfying $a\in P_A(\widetilde{b})$. The cone $\widehat{N}_B^{\tilde{A}}(b)$ at $b\in B$ is defined
analogously.
Then by \cite[Definition 6.1]{bauschke1}, specialized to the case of two sets, 
the CQ-number at $x^*$ associated with $(A,\widetilde A,B,\widetilde B)$ is
\[
\theta_\delta(A,\widetilde A,B,\widetilde B)
=
\sup\left\{ \langle u,v\rangle: u\in \widehat{N}_A^{\widetilde{B}}(a), -v\in \widehat{N}_B^{\widetilde{A}}(b),
\|u\|\leq 1, \|v\|\leq 1, a,b \in \ball(x^*,\delta)\right\}
\]
and the limiting CQ-number is
\[
\theta(A,\widetilde A,B,\widetilde B)=\lim_{\delta \to 0^+} \theta_\delta(A,\widetilde A,B,\widetilde B).
\]
The authors show in \cite[Theorem 6.8]{bauschke1} that for two sets  the CQ-condition 
$N_A^{\widetilde B}(x^*)\cap (-N_B^{\widetilde A}(x^*)) \subset\{0\}$,
implies $\theta(A,\widetilde A,B,\widetilde B) < 1$ . 

Using this, pick $\delta > 0$ such that $\theta_\delta(A,\widetilde A,B,\widetilde B) =: 1-\gamma  < 1$. 
Consider a building block $b\to a^+\to b^+$ as in definition \ref{separable} with $b,a^+,b^+\in U:=\ball(x^*,\delta)$. Then we have $b\in \widetilde B$ 
and $a^+\in \widetilde A$. Hence
$b-a^+\in \widehat{N}_A^{\widetilde B}(a^+)$ and $a^+-b^+\in \widehat{N}_B^{\widetilde A}(b^+)$, and also
$$u=(b-a^+)/\|b-a^+\|\in \widehat{N}_A^{\widetilde B}(a^+)\mbox{ and }v=(b^+-a^+)/\|b^+-a^+\| \in -\widehat{N}_B^{\widetilde A}(b^+).$$
Therefore, if
$\alpha = \angle(b-a^+,b^+-a^+)$, then $\cos\alpha=\langle u,v\rangle  \leq \theta_\delta(A,\widetilde  A,B,\widetilde B)=1- \gamma$ by the definition of $\theta_\delta$, 
because $b,a^+,b^+\in \ball(x^*,\delta)$ and $\|u\|=\|v\|=1$.
That shows $1-\cos\alpha \ge \gamma > 0$
and proves  that $B$ intersects $A$ $0$-separably at $x^*$ with constant $\gamma$. The estimate for building blocks $a\to b\to a^+$
is analogous.
\hfill $\square$
\end{proof}

Example \ref{example3} shows  that the converse of proposition
\ref{prop1} is not
true. In fact, 
$0$-separability
seems more versatile
in applications, while still guaranteeing linear convergence.  We  conclude by 
noting that linearly regular intersection in the sense of
\cite{luke}; and transversality in the sense of \cite{malick},
imply $0$-separability.  

Following \cite[section 2, (2.2)]{luke},  $A$ and $B$ have linearly regular intersection at $x^*\in A\cap B$ if
\begin{eqnarray}
\label{lin-reg}
N_A(x^*)\cap \left(-N_B(x^*)\right) =\{0\}.
\end{eqnarray}
This property is called {\it strong regularity} in \cite{kruger} and {\it the basic qualification condition for sets} in \cite[Definition 3.2 (i)]{mord}.
As a consequence of  $N_A^{\widetilde{B}}(a)\subset N_A(a)$, $N_B^{\widetilde A}(b)\subset N_B(b)$, linearly regular intersection 
implies that $(A,\widetilde A, B, \widetilde B$) satisfies the CQ-condition at $x^*$ for any nonempty  $\widetilde A$ and $\widetilde B$ in $\R^n$; cf. \cite{bauschke2}. 
By  Proposition \ref{prop1} we therefore have:
\begin{corollary}
\label{cor1}
{\bf (Linear regularity implies $0$-separability)}. Suppose $A,B$ intersect linearly
regularly at $x^*\in A\cap B$. 
Then they intersect $0$-separably at $x^*$.
\hfill $\square$
\end{corollary}

As we mentioned before, in the context of
alternating projections linear regularity and the CQ-condition are 
to be understood as transversality type hypotheses, indicating that the sets
$A,B$ intersect {\em at an angle at} $x^*$, as opposed to intersecting tangentially. This is confirmed by
relating $0$-separability to the classical notion of transversality. Following
\cite[def. 3]{malick}, two $C^2$-manifolds $A,B$ in $\mathbb R^n$ intersect transversally
at $x^*\in A \cap B$ if
\begin{eqnarray}
\label{trans}
T_A(x^*) + T_B(x^*) = \mathbb R^n,
\end{eqnarray}
where $T_M(x^*)$ is the tangent space to $M$ at $x^*\in M$. We then have the following

\begin{corollary}
Let $A,B$ be $C^2$-manifolds which intersect transversally at $x^*$.
Then $A$ and $B$ intersect $0$-separably at $x^*$. \hfill $\square$
\end{corollary}

\begin{proof}
Indeed, as shown in {\cite[Theorem 18]{malick}}, classical transversality (\ref{trans})
implies linear regular intersection (\ref{lin-reg}),  hence we can apply Corollary \ref{cor1}.
\hfill $\square$
\end{proof}

After the initial version \cite{initial} of this work was published, a related concept
termed intrinsic transversality was proposed in \cite{obsolete}.
Following \cite[Def. 2.2]{obsolete}, $A,B$ are intrinsically
transversal at $x^*\in A \cap B$ with  constant $\kappa \in (0,1]$ 
if there exists a neighborhood $U$
of $x^*$ such that for every $a^+\in A \cap U \setminus B$ and every
$b^+\in B \cap U \setminus A$ the estimate
\begin{eqnarray}
\label{intrinsic}
\max \left\{ d\left(\frac{a^+-b^+}{\|a^+-b^+\|},N_B^p(b^+)   \right), d\left( \frac{a^+-b^+}{\|a^+-b^+\|},-N_A^p(a^+)  \right) \right\} \geq \kappa >0
\end{eqnarray}
is satisfied. This relates to $0$-separability as follows.

\begin{proposition}
\label{hoax}
{\bf (Intrinsic transversality implies $0$-separability)}.
Suppose $A,B$ are intrinsically transversal at $x^*\in A \cap B$ with transversality constant $\kappa\in (0,1]$.
Then they intersect $0$-separably at $x^*$ with constant $\gamma=\kappa^2/2$.
\end{proposition}

\begin{proof}
By assumption there exists a neighborhood $U$ of $x^*$ on which
(\ref{intrinsic}) is satisfied.
Now let $b^+\in P_B(a^+)\cap U$, $a^+\not= b^+$, $a^+\in A \cap U$. Then
since $a^+-b^+\in N^p_B(b^+)$, we have $d\left( \frac{a^+-b^+}{\|a^+-b^+\|}, N^p_B(b^+)  \right)=0$,
so by (\ref{intrinsic}) we must have $d\left( \frac{a^+-b^+}{\|a^+-b^+\|}, -N^p_A (a^+)\right)\geq \kappa$.
Since $b-a^+\in N^p_A(a^+)$, we obtain
\[
2-2 \frac{\langle b^+-a^+,b-a^+\rangle}{\|a^+-b^+\| \|b-a^+\|} \geq \kappa^2,
\]
and this readily gives
\[
\langle b^+-a^+,b-a^+\rangle \leq \left( 1-\frac{\kappa^2}{2} \right) \|a^+-b^+\| \|b-a^+\|,
\]
which is (\ref{separable}) for the case $\omega = 0$ and $\gamma=\kappa^2/2$,  as claimed. 
\hfill $\square$
\end{proof}

We will resume the discussion of separable intersection of sets
in section \ref{loja}.

\section{H\"older regularity}
\label{hoelder}

In this section we introduce the concept of H\"older  regularity. 
We then relate it to other regularity notions like Clarke regularity,
prox-regularity, superregularity in the  sense of \cite{luke}, and its extension in \cite{bauschke1}.

\begin{definition}
{\bf (H\"older regularity)}.
Let $\sigma\in [0,1)$. The set $B$ is $\sigma$-H\"older regular with respect to $A$ at $b^*\in A \cap B$
if there exists a neighborhood $U$ of $b^*$ and a constant $c>0$ such that for every $a^+\in A\cap U$, 
and every $b^+\in P_B(a^+) \cap U$  one has
\begin{equation}
\label{empty}
\ball(a^+,(1+c)r) \cap \{b\in P_A(a^+)^{-1}: 
\langle a^+-b^+,b-b^+\rangle > \sqrt{c} r^{\sigma+1}\|b-b^+\|\} \cap B = \emptyset,
\end{equation} 
where $r = \|a^+-b^+\|$. We say that $B$ is H\"older regular with respect to $A$ if it is $\sigma$-H\"older
regular with respect to $A$ for every $\sigma\in [0,1)$.
\hfill $\square$
\end{definition}

\begin{remark}
Using the angle $\beta=\angle (a^+-b^+,b-b^+)$ and $r=\|a^+-b^+\|$,  condition (\ref{empty})
can be re-written in the following more suggestive form 
\begin{equation*}
   \makebox[\linewidth]{($\ref{empty}^\prime$) \hspace{0.4cm}$\displaystyle \ball(a^+,(1+c)r) \cap \{b\in P_A(a^+)^{-1} :
\cos \beta > \sqrt{c} r^\sigma\} \cap B = \emptyset.$ \hfill}
\end{equation*}
Geometrically, this means that the right circular cone with axis $a^+-b^+$ and aperture $\beta=\arccos \sqrt{c}r^\sigma$ 
truncated by the ball $\ball(a^+,(1+c)r)$  and the $B$-restricted proximal normal cone $\widehat{N}_A^B(a^+)$  contains
no points of $B$ other than $b^+$.
\end{remark}

In the  remainder of this section 
we  relate H\"older regularity to older geometric and analytic regularity concepts. 
We first consider  notions related to $0$-H\"older regularity. The case of $\sigma$-H\"older regularity
with $\sigma > 0$ will be considered later.

\begin{definition}
{\bf (Superregularity \cite[Proposition 4.4]{luke})}.
A closed set $B$ in $\mathbb R^n$ is called superregular at $b^*\in B$ if for every $\epsilon > 0$ there exists $\delta > 0$
such that for all $b,b^+\in \ball(b^*,\delta)\cap  B$ and $u\in N_B^p(b^+)$, the estimate
$\langle u,b-b^+\rangle \leq \epsilon \|u\| \|b-b^+\|$ is satisfied. 
\end{definition}

\begin{definition}
\label{def_bau}
{\bf ($(A,\epsilon,\delta)$-regularity \cite[Definition 8.1, (i)]{bauschke1})}.
Let $A,B$ be closed sets in $\mathbb R^n$.
$B$ is called $(A,\epsilon,\delta)$-regular at $b^*\in A \cap B$ if for all $b,b^+\in \ball(b^*,\delta)\cap B$
and every $u\in \widehat{N}_B^A(b^+)$, the estimate
$\langle u,b-b^+\rangle \leq \epsilon \|u\| \|b-b^+\|$ is satisfied.
\end{definition}

The two concepts are linked as follows:
$B$ is superregular at $b^*\in B$ if and only if for every $\epsilon > 0$ there exists $\delta>0$
such that $B$ is $(\mathbb R^n,\epsilon,\delta)$-regular at $b^*$ in the sense of Definition \ref{def_bau}, see \cite[Definition 8.1]{bauschke1}. 

\begin{proposition}
\label{regular}
{\bf ($0$-H\"older regularity from superregularity)}.
Suppose $B$ is $(A,\epsilon,\delta)$-regular at $b^*\in A \cap B$. Then $B$ is $0$-H\"older regular at $b^*$
with respect to $A$ with constant $c=\epsilon^2$.  In particular, if $B$ is superregular at $b^*$,
then $B$ is $0$-H\"older regular with respect to $A$ with constant $c$ that 
may be chosen arbitrarily small.
\end{proposition}

\begin{proof}
Since superregularity of $B$ at $b^*$  implies that for every $\epsilon>0$ there exists
$\delta >0$ such that $B$ is $(A,\epsilon,\delta)$-regular at $b^*$, \cite{bauschke1},
it remains to prove the first part of the statement. 

In order to check $0$-H\"older regularity, we have to provide a neighborhood $U$ of $b^*$ and $c>0$ such that
(\ref{empty}) is satisfied with $\sigma=0$.
We choose $U = \ball(b^*,\frac{\delta}{4(1+\epsilon^2)})$ and put $c=\epsilon^2$. 
To check (\ref{empty})  pick $a^+,b^+\in U$ such that 
$b^+\in P_B(a^+)$, $a^+\in A$. That gives $r=\|b^+-a^+\|\leq \frac{\delta}{2(1+c)}$.
By the definition of the restricted normal cone we have
$u:=a^+-b^+\in \widehat{N}_B^A(b^+)$. Now let $b\in B$, $b\not= b^+$. We have to show that
$b$ is not an element of the set in (\ref{empty}) for $\sigma=0$.
Suppose $b\in \ball(a^+,(1+c )r)$. 
Then we have to show $\langle a^+-b^+,b-b^+\rangle \leq \sqrt{c}r\|b-b^+\|$. Observe that
$\|b-b^*\| \leq \|b-a^+\|+\|a^+-b^*\| \leq (1+c )r + \frac{\delta}{{4}(1+c)}\leq (1+c ) \frac{\delta}{2(1+c)}+   \frac{\delta}{4(1+c)} < \delta$.
Hence
$(A,\epsilon,\delta)$-regularity  implies $\langle u,b-b^+\rangle \leq \epsilon \|u\| \|b-b^+\| = \sqrt{c}\|u\| \|b-b^+\|$, and the claim follows.
\hfill $\square$
\end{proof}

\begin{remark}
Example \ref{packman} shows  that the converse of proposition \ref{regular} is not true. 
The difference between superregularity and its extension
$(A,\epsilon,\delta)$-regularity on the one hand, and $0$-H\"older regularity on the other,
 is the following:  in (\ref{empty})  we exclude
points in the intersection
of a restricted right circular cone with vertex $b^+$, axis $a^+-b^+$, and aperture $\beta  = \arccos \sqrt{c}r^\sigma$ and the shrinking ball
$\ball(a^+,(1+c)r)$. In contrast,  $(A,\epsilon,\delta)$-regularity forbids many more points,
namely all points in that same 
cone, but within the fixed ball $\ball(b^*,\delta)$.  In consequence, this type
of regularity is  not suited to deal with singularities pointing inwards, 
like the prototype in example \ref{packman}. 
\hfill $\square$
\end{remark}

\begin{remark}
If $B$ is $\sigma$-H\"older regular at $b^*$ with respect to $A$
with constant $c>0$ on the neighborhood $U$ of $b^*$, and if $\sigma' < \sigma$, then
for every $c'\in (0,c)$  there exists
a neighborhood $V \subset U$ of $b^*$ such that $B$ is $\sigma'$-H\"older regular at
$b^*$ with constant $c'$.
Indeed, if $b\in \ball(a^+,(1+c')r)$ in (\ref{empty}), then also $b\in \ball(a^+,(1+c)r)$, hence
by assumption
$\cos\beta \leq \sqrt{c}r^\sigma = \sqrt{c} r^{\sigma-\sigma'}r^{\sigma'} \leq \sqrt{c'}r^{\sigma'}$
if $V$ is chosen so that $\sqrt{c}r^{\sigma-\sigma'} < \sqrt{c'}$.
\hfill $\square$
\end{remark}

We next justify our notion of H\"older regularity by 
proving that prox-regular sets are $\sigma$-H\"older regular for every $\sigma\in (0,1]$. Recall  that
a set $B$ in $\mathbb R^n$ is {\em prox-regular} at $b^*\in B$ if there exists
a neighborhood $U$ of $b^*$ such that $P_B(y)$ is single-valued for every $y\in U$, cf. \cite[Chapter 13]{rock}.
 
Consider $b\in B$ and let
$d\in N_B^p(b)$ be  a unit proximal normal to $B$ at $b$.  Define the reach of
$B$ at $b$ along $d$ as
\begin{eqnarray}
\label{RR}
R(b,d) = \sup\{R \geq 0: b = P_B(b+td) \mbox{ for every } 0 \leq t \leq R\}.
\end{eqnarray}
Then $R(b,d) \in (0,\infty]$, and
the case $R(b,d)=\infty$ occurs e.g. if $B$ is convex and $b$ a boundary point of $B$.
We can say that $\ball(b+R(b,d)d,R(b,d))$ is the largest ball with its centre on
$b+\mathbb R_+d$ which touches $B$ in $b$ from outside, i.e., has no points from $B$ in its interior. 

It was shown in \cite[Thm. 1.3 $(h)$]{thibault} that $B$ is prox-regular at $b^*\in B$
if and only if there exists $r>0$ and a neighborhood $U$ of $b^*$ such that $R(b,d)\geq r$
for every $b\in U \cap B$ and every $d\in N_B^p(b)$ with $\|d\|=1$. 
An immediate consequence is that sets of positive reach
in the sense of Federer \cite{federer} are prox-regular;  see e.g. \cite[Theorem 1.3]{thibault}. 
Therefore,  prox-regularity 
is a local version 
of positive, or non-vanishing,  reach.

We now relax the concept of non-vanishing reach to sets where the reach may vanish at some boundary points,
but slowly so.

\begin{definition}
Let $\sigma\in (0,1]$.
The set $B$ has $\sigma$-slowly vanishing reach with respect to the set $A$ at $b^*\in A\cap B$
if there exists $0 \leq \tau < 1$ such that
\begin{eqnarray}
\label{slow}
\limsup_{A \ni a\to b^*, b\in P_B(a)} \frac{\|a-b\|^\sigma}{R(b,d)} \leq \tau,
\end{eqnarray}
where 
$d=(a-b)/\|a-b\|$. 
We say that the reach vanishes with exponent $\sigma$ and
rate $\tau$.
\hfill $\square$
\end{definition}

\begin{proposition}
\label{prox}
If $B$ is prox-regular at $b^*\in A\cap B$, then it has slowly vanishing reach at $b^*$ with respect to  $A$
with rate $\tau=0$ and arbitrary exponent $\sigma\in (0,1]$.
\end{proposition}

\begin{proof}
Let $\tau' > 0$.
By \cite[Thm. 1.3 $(h)$]{thibault} prox-regularity at $b^*$ implies that there exist $\epsilon > 0$ and $r > 0$ such that $R(b,d) \ge r$ for every 
$b\in B$ with $\|b-b^*\|\leq \epsilon$ and every $d\in N_B^p(b)$ with $\|d\|=1$.  By shrinking $\epsilon$
if necessary, we may assume $\epsilon^\sigma/(2^\sigma r) < \tau'$.
Now let $a\in A \cap \ball(b^*,\frac{\epsilon}{2})$ be arbitrary, choose $b\in P_B(a)$, and let $d=(a-b)/\|a-b\|$. 
Then as $b^*\in A \cap B$, we have $\|b-b^*\|\leq \|b-a\|+\|a-b^*\|\leq 2\|a-b^*\|\leq \epsilon$. Since
$d\in N_B^p(b)$, the above gives us
$R(b,d)\geq r$. Therefore, 
since $\|a-b\| = d_B(a) \leq \|a-b^*\|$, the quotient
in (\ref{slow}) satisfies 
\[
\frac{\|a-b\|^\sigma}{R(b,d)} \leq \frac{\epsilon^\sigma}{2^\sigma r} < \tau',
\]
and since $\tau' > 0$ was arbitrary, this shows that
(\ref{slow}) is satisfied with $\tau=0$.
\hfill $\square$
\end{proof}

\begin{proposition}
\label{reach}
{\bf (H\"older regularity from slowly vanishing reach)}.
Let $\sigma\in {(0,1)}$.
Suppose $B$ has $\sigma$-slowly vanishing reach with rate $\tau \in [0,1)$ with respect to $A$ at $b^*\in A\cap B$. Then
$B$ is $(1-\sigma)$-H\"older regular with respect to $A$ 
with any constant $c>0$
satisfying 
\begin{eqnarray}
\label{bad}
\frac{\tau}{2} \sqrt{2+c} < 1.
\end{eqnarray}
In particular, $c$ may be chosen arbitrarily small.
\end{proposition}

\begin{proof}
1)
We have to show that there exists a neighborhood $U$ of $b^*$ such that (\ref{empty})
is satisfied with $c$ as in (\ref{bad}) and with exponent $1-\sigma$. 

By condition  (\ref{bad}) we can choose $\tau' > \tau$ and $\epsilon > 0$ such that
\[
\frac{\tau'}{2} \left(  \epsilon + \sqrt{\epsilon^2 + 2 + c}\right) < 1.
\]
By condition (\ref{slow}), and since $\tau < \tau'$, there exists a neighborhood $U$ of $b^*$ such that whenever $a^+,b^+\in U$,
$a^+\in A$, $b^+\in P_B(a^+)$ and $d=(a^+-b^+)/\|a^+-b^+\|$, then $r^\sigma/R(b^+,d) < \tau'$, where $r:= \|a^+-b^+\|$.
On shrinking $U$ further if necessary, we may arrange that $a^+,b^+\in U$ implies $r^{1-\sigma}=\|a^+-b^+\|^{1-\sigma} < \epsilon$.
We will show that $U$ is the neighborhood we need in condition (\ref{empty}).

2)
To prove this pick $a^+,b^+\in U$, $a^+\in A$, $b^+\in P_B(a^+)$, put $r=\|a^+-b^+\|$,
and let $b\in B$, $b\not= b^+$. We have to show that $b$ is not an element of the set (\ref{empty}$^\prime$). To check this,
let $\beta$ be the angle $\beta = \angle(a^+-b^+,b-b^+)$. Since there is nothing to prove  for $b\not\in \ball(a^+,(1+c)r)$,
we assume $b\in \ball(a^+,(1+c)r)$. Now we have to show that $\cos\beta \leq \sqrt{c}r^{1-\sigma}$. As this is clear for
$\cos\beta \leq 0$, we may assume $\cos\beta > 0$.

Let us define
\begin{equation}
\label{R}
R := \frac{r}{2} \left(1 +  \sqrt{1+\frac{2c+c^2}{\cos^2\beta}} \right),
\end{equation}
where $r,\beta$ are as before. We claim that the ball $\ball(b^++Rd,R)$ contains $b$, where as above
$d=(a^+-b^+)/\|a^+-b^+\|$. To prove this, note that by the cosine theorem, applied
in the triangle $a^+,b^+,b$,  we have
\[
\|a^+-b\|^2 = r^2 + \|b-b^+\|^2 - 2 r \|b-b^+\| \cos\beta.
\]
Since $\|a^+-b\|\leq (1+c)r$, we obtain
\[
r^2 + \|b-b^+\|^2 - 2r\|b-b^+\|\cos\beta \leq (1+c)^2 r^2,
\]
which on completing squares turns out  to be the same as
\[
\|b-b^+\| \leq r \left( \cos \beta + \sqrt{2c+c^2+\cos^2\beta} \right)=2R\cos\beta.
\]
Here the last equality uses the definition (\ref{R}) of $R$. We therefore obtain
\[
\|b-b^+\|^2 \leq 2R \cos\beta \|b-b^+\|, 
\]
and using the cosine theorem again, now in the triangle
$b^++Rd,b^+,b$, we deduce
\[
\|b^++Rd - b\|^2 = R^2 + \|b-b^+\|^2 -2R\|b-b^+\|\cos\beta
\leq R^2.
\]
This gives  $b\in \ball(b^++Rd,R)$ as claimed.

3) 
By the definition (\ref{RR}) of $R(b^+,d)$, any radius $R' < R(b^+,d)$
must give rise to a ball with $\ball(b^++R'd,R') \cap B = \{b^+\}$.  But as we have shown in part 2), the ball 
$\ball(b^++Rd,R)$ contains $b$,  so necessarily $R \geq R(b^+,d)$.
Hence by the choice of $U$
in part 1), $r^\sigma/R \leq r^\sigma/R(b^+,d) < \tau'$, or what is the same, $r^\sigma < R\tau'$.
Substituting the definition  (\ref{R}) of $R$ and multiplying by $r^{-\sigma}$, we deduce
\[
1 < r^{1-\sigma} \tau' \left( \frac{1}{2}+\frac{1}{2}\sqrt{1+ \frac{2c+c^2}{\cos^2\beta}} \right).
\]
Now suppose that $\cos\beta > \sqrt{c}r^{1-\sigma}$, contrary to what we wish to show. Then
\begin{align*}
1 &< r^{1-\sigma} \tau'\left( \frac{1}{2} + \frac{1}{2} \sqrt{1+ \frac{2c +c^2}{c r^{2(1-\sigma)}}} \right)\\
&= \frac{\tau'}{2} \left( r^{1-\sigma} + \sqrt{r^{2(1-\sigma)}+2+c}\right) \\
&<  \frac{\tau'}{2} \left( \epsilon + \sqrt{\epsilon^2+2+c}\right) < 1,
\end{align*}
a contradiction. 
That proves the result. \hfill $\square$
\end{proof}

Since prox-regularity at $b^*\in B$ implies slowly vanishing reach at $b^*$ with respect to any closed set $A$ containing $b^*$,
we have the following immediate consequence.

\begin{corollary}
{\bf (H\"older regularity from prox-regularity)}.
Let $B$ be prox-regular. Then $B$ is $\sigma$-H\"older regular for every $\sigma\in [0,1)$
with a constant $c>0$ that may be chosen arbitrarily small.
\end{corollary}

\begin{proof}
For $\sigma =0$ this follows from Proposition \ref{regular},  
because prox-regularity implies superregularity.
For $\sigma\in (0,1)$
we obtain it by combining  Propositions \ref{prox} and \ref{reach}.
\hfill $\square$
\end{proof}

Consider the case of a Lipschitz domain $B$. 
Here H\"older regularity may be related to a property of the boundary $\partial B$ of $B$.

\begin{proposition}
\label{h_smooth}
Let $\sigma\in (0,1)$. Suppose $B$ is the epigraph of a locally Lipschitz
function
$f:\mathbb R^{n-1} \to \mathbb R$. Let $x^*\in \mathbb R^{n-1}$ 
and suppose there exists a neighborhood $V$
of $x^*$ and $\mu>0$ such that for every $x_0\in V$ and every proximal subgradient $g\in \partial_pf(x_0)$
the one-sided H\"older estimate $f(x_0)+\langle g,x-x_0\rangle - \mu \|x-x_0\|^{1+\sigma}\leq f(x)$ is satisfied for all $x\in V$.
Then $B$ is $\sigma$-H\"older regular at $(x^*,f(x^*))\in B$ with respect to every closed set $A$
containing $(x^*,f(x^*))$, and for every constant $c>0$ satisfying $\mu \leq \sqrt{c}/(2+c)^\sigma$.
\end{proposition}

\begin{proof}
We have to find a neighborhood $U$ of $b^*=(x^*,f(x^*))\in B$ such that (\ref{empty})
is satisfied with exponent $\sigma$ and constant $c$ satisfying $\mu \leq \sqrt{c}/(2+c)^\sigma$. Choose
$\epsilon > 0$ such that $\ball(x^*,\epsilon) \subset V$. Now choose $\delta>0$
with $\delta < \epsilon/(2+c)$ and define
$U = \ball(b^*,\delta)$. We will show that $U$ is as required.

In order to check (\ref{empty}), choose $a^+\in A\backslash B$ and $b^+\in P_B(a^+)$ such that $a^+,b^+\in U$. 
As $b^+\in P_B(a^+)\cap U$, we get  $b^+ =(x_0,f(x_0))\in B$ for some $x_0\in V\subset$ $\mathbb R^{n-1}$, while $a^+\not\in B= {\rm epi} f$ implies
$a^+=(x_1,\xi_1)$ for some $\xi_1 < f(x_1)$. Since $a^+-b^+$ is a proximal normal to $B$ at $b^+$, there exists
a proximal subgradient $g\in \partial_p f(x_0)$ such that $(x_1,\xi_1) = (x_0,f(x_0)) + t(g,-1)$ for some $t>0$. Using
$\|a^+-b^+\|=r$, we can therefore write
\[
a^+-b^+
=\left(  \frac{rg}{\sqrt{1+\|g\|^2}}, -\frac{r}{\sqrt{1+\|g\|^2}} \right).
\]
Now consider $b\in B$, $b\neq b^+$ such that $\|a^+-b\|\leq (1+c)r$. To verify (\ref{empty}$^\prime$) we have to show that
$\cos\beta \leq \sqrt{c} r^\sigma$, where $\beta=\angle(a^+-b^+,b-b^+)$. Since there is nothing to prove
for $\cos\beta \leq 0$, we assume $\cos\beta > 0$.
By the definition of $B$ we have
$b = (x,\xi)$ for some $x\in \mathbb R^{n-1}$ and  $\xi \geq f(x)$.  Now
\begin{align*}
\cos\beta &= \frac{\langle g,x-x_0\rangle - \xi+f(x_0)}{\sqrt{1+\|g\|^2}\sqrt{\|x-x_0\|^2 + \left( \xi-f(x_0)\right)^2}}\\
&\leq \frac{\langle g,x-x_0\rangle - f(x)+f(x_0)}{\|x-x_0\|} \\
&\leq \frac{\mu \|x-x_0\|^{1+\sigma}}{\|x-x_0\|}= \mu \|x-x_0\|^\sigma \leq \sqrt{c} r^\sigma.
\end{align*}
Here the first inequality uses the fact that $\xi \geq f(x)$.
 The second inequality uses the one-sided H\"older 
estimate from the hypothesis. In order to be allowed to use this estimate,
we have to assure that $x\in V$. This follows from
\begin{align*} 
\|x-x^*\| &\leq  \|b-b^*\| \leq \|b-a^+\| + \|a^+-b^*\| \\
&\leq  (1+c)\|a^+-b^+\| + \|a^+-b^*\| \\
&\leq (2+c)\|a^+-b^*\|\leq (2+c)\delta < \epsilon.
\end{align*}
The third inequality
can be seen as follows. We have
\begin{align*}
\|x-x_0\| &\leq  \|b-b^+\| \leq \|b-a^+\| + \| a^+-b^+\| \\
&\leq  (2+c) \|a^+-b^+\| = (2+c)r.
\end{align*}
Hence
\begin{align*}
\mu \|x-x_0\|^\sigma \leq \mu (2+c)^\sigma r^\sigma \leq \sqrt{c}r^\sigma
\end{align*}
by the choice of $c$.
That completes the argument.
\hfill $\square$
\end{proof}

\begin{remark}
The nomenclature in Proposition \ref{h_smooth} can be explained as follows. Lipschitz smoothness \cite{fabian}
of $-f$ at $x_0$ is a well-known second-order property
equivalent to
the second difference quotient
\[
\Delta_2(x) = \frac{f(x) - f(x_0)-\langle g,x-x_0\rangle}{\|x-x_0\|^2} \geq -\mu > -\infty
\]
being bounded below for $g\in \partial f(x_0)$ and $x$
in a neighborhood of $x_0$. 
The H\"older estimate
in Proposition \ref{h_smooth} is  the analogous but  weaker condition $\Delta_{1+\sigma}(\cdot) \geq -\mu > -\infty$
for some $\sigma\in (0,1)$.  In analogy with \cite{fabian} one could call this $\sigma$-H\"older smoothness of $-f$ at $x_0$.
\end{remark}

We consider the following natural
modification of 
amenability  from \cite{rock}. The set $B\subset \mathbb R^n$  is called  $\sigma$-H\"older amenable at $x^*\in B$
if there exists a neighborhood $U$ of $x^*$, a class $C^{1,\sigma}$-mapping
$G:\mathbb R^n \to \mathbb R^m$, and a closed convex set $C\subset \mathbb R^m$,
such that $B \cap U =\{x\in U: G(x) \in C\}$ and
$N_C\left( G(x^*)  \right) \cap {\rm ker}\left( DG(x^*)^T  \right) = \{0\}$, where $DG(x^*)$ denotes the first-order differential of $G$ at $x^*$.
A typical example  in optimization is when $B$ is defined by 
$C^{1,\sigma}$ equality and inequality constraints, where the Mangasarian-Fromowitz constraint qualification holds
at $x^*$,  see \cite[Prop. 2.3]{thibault}, \cite[Chap 10, F.]{rock}, \cite[Prop. 4.8]{luke}.

\begin{proposition}
{\bf (H\"older regularity from H\"older amenability)}.
Suppose the closed set $B$ is $\sigma'$-H\"older amenable
at $x^*$ for some $\sigma'\in (0,1]$. Then $B$ is $\sigma$-H\"older regular
at $x^*$ with respect to any closed set $A$ containing $x^*$  for every $\sigma\in (0,\sigma')$, and with arbitrary constant $c$.
\hfill $\square$
\end{proposition}

The proof may be  adopted from on \cite[Prop. 4.8]{luke} with minor changes, and we skip the details. 
This result suggests that H\"older regularity is settled between the weaker superregularity
and the stronger prox-regularity. This is true as long as we consider this type
of regularity as a property of $B$
alone.  We stress, however,   that it is the combination with $A$ and the shrinking distance between
the  sets in (\ref{empty}) which makes our definition \ref{regular} truly versatile in applications. 
This is corroborated by the following observation.

\begin{proposition}
\label{kill}
{\bf (H\"older regularity from intrinsic transversality)}.
Suppose $A$, $B$ are intrinsically transversal at $x^*$ 
with constant of
transversality $\kappa\in (0,1]$. Then  $B$ is
$\sigma$-H\"older regular at $x^*$ with respect to $A$ for every $\sigma\in [0,1)$, with any constant  $c>0$
satisfying $c < \frac{\kappa^2}{1-\kappa^2}$. 
\end{proposition}

\begin{proof}
1)
By \cite[Def. 2.2]{obsolete} and \cite[Prop. 6.9]{obsolete} there exists a neighborhood
$V$ of $x^*$ such that 
max$\left\{ {\rm dist}(u, N_B(b)),{\rm dist}(u,- N_A(a^+))\right\} \geq \kappa>0$
for all $a^+\in A\cap V \setminus B$, $b\in B\cap V\setminus A$, and $u=(a^+-b)/\|a^+-b\|$. 
Following entirely the argument in \cite[page 6]{obsolete},
one can now find a smaller neighborhood $U$ of $x^*$
such that the following is true: If $b\in B \cap U$ and $a^+\in P_A(b)\cap U$,
then even $d_B(a^+) \leq (1-\kappa^2) \|b-a^+\|$.

2)
We claim that $U$ is the neighborhood required in $\sigma$-H\"older regularity with constant $c$. To check this, we have to show that the set (\ref{empty})
is empty. We assume that $b\in U$ is an element of that set. Then $b\in P_A(a^+)^{-1}\cap B$ and $b\in \ball(a^+,(1+c)r)$. Hence we are in the situation
of part 1), which means $r=d_B(a^+) \leq (1-\kappa^2)\|b-a^+\| \leq (1-\kappa^2)(1+c) r < r$, a contradiction.
Hence the set (\ref{empty}) is empty.
\hfill $\square$
\end{proof}

\section{Convergence}
\label{convergence}
In this section we prove the main convergence result. Alternating projections converge
locally for sets which intersect separably,  
if one of the sets is H\"older regular with respect to the other. The proof requires the following preparatory lemma.

\begin{lemma}
\label{third}
{\bf (Three-point estimate)}.
Suppose $B$ intersects $A$ separably at $x^*\in A\cap B$ with exponent
$\omega\in [0,2)$ and constant  $\gamma >0$ on the neighborhood $U$ of $x^*$.
Suppose 
$B$ is also $\omega/2$-H\"older regular
at $x^*\in A\cap B$ with respect to $A$ on $U$ with constant $c>0$ satisfying
$c < \frac{\gamma}{2}$. 
Then there exists   $0<\ell< 1$, depending only on $\gamma,c$ and $U$,  such that
\begin{eqnarray}
\label{strong}
\|a^+-b^+\|^2 + \ell \|b-b^+\|^2 \leq \|a^+-b\|^2
\end{eqnarray}
for every building block $b\to a^+\to b^+$ in $U$.
\end{lemma}

\begin{proof}
1)
By the cosine theorem we have
\[
\|a^+-b\|^2 = \|a^+-b^+\|^2 + \|b-b^+\|^2 - 2 \|a^+-b^+\|\|b-b^+\| \cos\beta,
\] 
where $\beta = \angle(b-b^+,a^+-b^+)$.
Hence in order to assure (\ref{strong})
we have to find  $\ell\in (0,1)$ such that
\begin{eqnarray}
\label{must}
\textstyle\frac{1-\ell}{2} \|b-b^+\| \geq \|a^+-b^+\| \cos \beta
\end{eqnarray}
for all building blocks $b\to a^+\to b^+$ in $U$.
We consider two cases. Case I is when
$\beta \in ( \frac{\pi}{2},\pi]$. Case II is $\beta \in [0,\frac{\pi}{2}]$.

2)
We start by discussing case I.
For angles $\beta \in (\frac{\pi}{2},\pi]$ we have $\cos\beta < 0$, hence
(\ref{must})  is trivially true if we choose any
$0 < \ell  < 1$. For instance $\ell = \frac{1}{2}$ will do. To establish
(\ref{must}) we may now concentrate on case II,
where  $\beta \in [0,\frac{\pi}{2}]$.

3) In case II  we want to use $\omega/2$-H\"older regularity of $B$ with respect to $A$. 
We subdivide case II in two subcases. Case IIa is when $\cos \beta\leq \sqrt{c}\|a^+-b^+\|^{\frac{\omega}{2}}$. 
Case IIb is when $\cos \beta>\sqrt{c}\|a^+-b^+\|^{\frac{\omega}{2}}$.

Let us start with case IIa, where $\cos\beta\leq \sqrt{c}\|a^+-b^+\|^{\frac{\omega}{2}}$.  Observe that
\begin{align*}
\|b-b^+\|^2 &= \|b-a^+\|^2 +\|a^+-b^+\|^2-2\langle b-a^+,b^+-a^+\rangle\\
&= (\|b-a^+\|-\|a^+-b^+\|)^2 +2\|b-a^+\|\|a^+-b^+\|\left(1-\cos \alpha\right)\\
&\geq 2\|b-a^+\|\|a^+-b^+\|\left(1-\cos  \alpha \right),
\end{align*}
where  $\alpha = \angle (b-a^+,b^+-a^+)$. By the angle condition $(\ref{angle}')$  we have
$1-\cos \alpha \geq \gamma\|a^+-b^+\|^\omega$ for every building block $b\to a^+\to b^+$ in $U$. Hence
\begin{eqnarray*}
\|b-b^+\|^2 \geq 2\gamma \|b-a^+\|\|a^+-b^+\|^{1+\omega}
\geq 2\gamma \|a^+-b^+\|^{2+\omega},
\end{eqnarray*}
where the last estimate uses $b^+\in P_B(a^+)$. Altogether we obtain
\begin{eqnarray*}
\|b-b^+\|\geq \sqrt{2\gamma} \|a^+-b^+\|^{1+\frac{\omega}{2}}
\geq   \sqrt{\frac{{2\gamma}}{c}} \|a^+-b^+\|\cos \beta,
\end{eqnarray*}
bearing in mind that we are in case IIa. To assure (\ref{must})
we put $\ell = 1 - \sqrt{\frac{2c}{\gamma}}$.
Then $\ell\in (0,1)$,  because of the hypothesis  $c<\frac{\gamma}{2}$.

4)
Let us now deal with case IIb, where $\cos \beta >\sqrt{c}\|a^+-b^+\|^{\frac{\omega}{2}}$. 
By H\"older regularity (\ref{empty}) of  $B$ with respect to $A$ and since $a^+\in P_A(b)$, we have 
$b\not\in \ball(a^+,(1+c )r)$, where $r=\|a^+-b^+\|$. In other words, $\|b-a^+\|  > (1+c)\|a^+-b^+\|$. 
Using this and the cosine theorem again, we have
\begin{align*}
\|b-b^+\|^2 &= \|b-a^+\|^2 - \|a^+-b^+\|^2+2 \|b-b^+\| \|a^+-b^+\|\cos \beta\\
&\geq  c (c +2)\|a^+-b^+\|^2+2 \|b-b^+\| \|a^+-b^+\|\cos \beta.
\end{align*}
Since $a^+\not = b^+$,
this may be rearranged as
\begin{eqnarray}
\label{arranged}
\displaystyle\frac{\|b-b^+\|^2}{\|a^+-b^+\|^2} -2\frac{\|b-b^+\|}{\|a^+-b^+\|}\cos\beta - c (c  +2)\geq 0.
\end{eqnarray}
Hence (\ref{arranged}) implies that
the polynomial $P(X)=X^2-2X\cos \beta-c (c +2)$ is nonnegative at $X=\frac{\|b-b^+\|}{\|a^+-b^+\|}$. 
But for nonnegative $X$,   nonnegativity $P(X)\geq 0$ is equivalent to
\begin{eqnarray*}
X\geq  \cos \beta +\sqrt{\cos^2 \beta+c (c +2)} = \cos \beta \left(1+\sqrt{1+\frac{c  (c  +2)}{\cos^2 \beta}}\right).
\end{eqnarray*}
Hence for $X=\frac{\|b-b^+\|}{\|a^+-b^+\|}$ we  know that
\begin{eqnarray}
\label{holds}
\frac{\|b-b^+\|}{\|a^+-b^+\|} \geq\cos \beta  \left(1+\sqrt{1+\frac{c (c +2)}{\cos^2\beta}}\right).
\end{eqnarray}
Put $\Theta_{r,\beta} = 1+ \sqrt{1+\frac{c(c +2)}{\cos^2\beta}}$, then
$\Theta_{r,\beta} \geq c+2$.  
Hence $\ell = \frac{c}{2+c}$ assures (\ref{must})  in case IIb.

5) In conclusion, if we put $\ell = \min\left \{\frac{1}{2}, 1-\sqrt{\frac{2c}{\gamma}} ,\frac{c}{2+c} \right\}$, 
with $c<\frac{\gamma}{2}$, then (\ref{must}) is satisfied in all cases. 
\hfill $\square$
\end{proof}

\begin{theorem}
\label{theorem1}
{\bf (Local convergence)}.
Suppose $B$ intersects $A$ separably at $x^*\in A\cap B$ with exponent $\omega\in [0,2)$ and constant $\gamma$
and   is $\omega/2$-H\"older regular at $x^*$ with respect to $A$ and constant $c < \frac{\gamma}{2}$. 
Then there exists a neighborhood $V$ of $x^*$ such that every sequence of alternating projections between $A$ and $B$ which enters
$V$,  converges to a point $b^*\in A \cap B$.
\end{theorem}

\begin{proof}
1)
By hypothesis there exists a neighborhood
$U=\ball(x^*,4\epsilon)$ of $x^*\in A \cap B$ such that  
every building block $b\to a^+\to b^+$ with $b,a^+,b^+\in U$ satisfies the angle
condition $1-\cos\alpha \geq \gamma \|b^+-a^+\|^\omega$, 
where $\alpha = \angle(b-a^+,b^+-a^+)$. In addition, by shrinking $U$ if necessary,
we may assume 
that  $B$ is $\omega/2$-H\"older regular at $x^*$ on $U$ with constant $c < \frac{\gamma}{2}$.
Then by the three-point estimate (Lemma \ref{third}) there exists $\ell\in (0,1)$
depending only on $c,\gamma$ and $U$, such that
$\|a^+-b^+\|^2 + \ell \|b-b^+\|^2 \leq \|a^+-b\|^2$ for every such building block. Since
$\|a^+-b\| \leq \|a-b\|$, we deduce the following four-point estimate
\begin{eqnarray}
\label{four}
d_B(a^+)^2 + \ell \|b-b^+\|^2 \leq d_B(a)^2 
\end{eqnarray}
for building blocks $a\to b\to a^+\to b^+$  with  $b,a^+,b^+\in U$.

2)
Define the constants $\theta = (\omega+2)/4$ and $C=1/((1-\theta)\ell \sqrt{2\gamma}))$.
Choose $\delta > 0$ such that
\begin{center}
$9\delta+C 2^{2(1-\theta)}\delta^{2(1-\theta)} < \frac{\epsilon}{4}$,
\end{center}
which implies  $16\delta<\epsilon$. Then define the neighborhood $V$ as $V=\ball(x^*,\delta)$.
We have to show that if the alternating sequence enters $V$,
then it converges to a unique limit $b^*\in A \cap B$. 
By relabeling the sequence, we may without loss of generality assume that
$b_0\in V = \ball(x^*,\delta)$.
The case where the $a_k$'s reach
$V$ first  is treated analogously.

We shall prove by induction
that for every $k\geq 1$, we have
\begin{eqnarray}
\label{ind1}
b_k, 
a_{k+1}, b_{k+1} \in \ball(x^*,\epsilon)
\end{eqnarray}
and
\begin{eqnarray}
\label{ind2}
\sum_{{j=1}}^k \|b_j-b_{j+1}\| \leq \frac{1}{2} \sum_{{j=1}}^k \|b_{j-1}-b_j\| + \frac{C}{2} \left( d_B({a_{1}})^{2(1-\theta)}
- d_B(a_{k+1})^{2(1-\theta)} \right).
\end{eqnarray}

Let us first do the induction step and suppose that hypotheses (\ref{ind1}), (\ref{ind2})
are true at $k-1$  for some $k\geq 2$.  We have to show that they also hold at $k$. 

2.1) 
Firstly we check (\ref{ind1}) at $k$. By (\ref{ind1}) at $k-1$ we know that $b_{k-1},a_k,b_k\in \ball(x^*,\epsilon)$. So  it remains
to prove $a_{k+1},b_{k+1}\in \ball(x^*,\epsilon)$.
We claim that $b_k\in \ball(x^*,\frac{\epsilon}{4})$. Indeed, using the induction hypothesis (\ref{ind2}) at $k-1$  gives
$$
\displaystyle \sum_{{j=1}}^{k-1} \|b_j-b_{j+1}\| \leq  \frac{1}{2}\sum_{{j=1}}^{k-1} \|b_j-b_{j-1}\| +\frac{C}{2} \left( d_B({a_{1}})^{2(1-\theta)} - d_B(a_{k})^{2(1-\theta)} \right).
$$
Hence
\begin{align*}
\displaystyle \sum_{{j=1}}^{k-1} \|b_j-b_{j+1}\| &\leq \|{b_{0}}-{b_{1}}\| + C\,d_B({a_{1}})^{2(1-\theta)} - C\,d_B(a_{k})^{2(1-\theta)} -\|b_{k-1}-b_{k}\|  \nonumber \\
&\leq    \|{b_{0}-b_{1}}\| + C\,d_B({a_{1}})^{2(1-\theta)}.\\ \nonumber
\end{align*}
Therefore,
\begin{align}
\label{etape}
\|b_{k} - x^*\| &\leq \|b_{k}- {b_{1}}\| + \|{b_{1}}-x^*\|\nonumber
\leq \sum_{{j=  1}}^{k-1} \|b_{j+1}-b_{j}\| + \|{b_{1}}-x^*\|\\
&\leq   \|{b_{0}-b_{1}}\| + C\,d_B({a_{1}})^{2(1-\theta)}+ \|{b_{1}}-x^*\|.
\end{align}
Since $b_0\in \ball(x^*,\delta)$, we have $\|b_{0}-a_{1}\| \leq \|b_0-x^*\| \leq  \delta$, 
hence $a_{1}\in \ball(x^*,2\delta).$
Then $\|{a_{1}-b_{1}}\| \leq \|{a_{1}}-x^*\| \leq 2 \delta$, which gives 
$\|{b_{1}} - x^*\|\leq 4 \delta.$
It follows that
$\|{b_{0}-b_{1}}\| \leq 5 \delta.$
Now since $d_B({a_{1}})=\|{a_{1}-b_{1}}\| \leq 2\delta$, going back to (\ref{etape}), we obtain
$$\|b_{k} - x^*\| \leq 9\delta + C2^{2(1-\theta)}\delta^{2(1-\theta)}<\frac{\epsilon}{4},$$
{which is our above claim. Now this implies}
$$\|a_{k+1}-x^*\| \leq \|a_{k+1}-b_{k}\| +\|b_k-x^*\|\leq 2\|b_k-x^*\|<\frac{\epsilon}{2}<\epsilon,$$
and
$$\|b_{k+1}-x^*\| \leq \|a_{k+1}-b_{k+1}\| +\|a_{k+1}-x^*\|\leq 2\|a_{k+1}-x^*\|<\epsilon.$$
This proves  $a_{k+1}\in \ball(x^*,\epsilon)$ and $b_{k+1}\in \ball(x^*,\epsilon)$ {and therefore (\ref{ind1}) at $k$}.

2.2) Let us now prove that (\ref{ind2}) is true at $k$. Using the induction hypothesis 
{(\ref{ind1}) at $k-1$, that is,}
$b_{k-1}$, $a_k$, $b_k\in \ball(x^*,\epsilon)$, we apply part 1)  of the proof
to the building block $b_{k-1}\to a_k\to b_k$, which gives
\begin{eqnarray}
\label{Angle}
\frac{1-\cos\alpha_k}{\|a_k-b_k\|^\omega}\geq \gamma,
\end{eqnarray}
where $\alpha_k=\angle (b_{k-1}-a_k,b_k-a_k)$.
By part 2.1), which is already proved, we have 
$b_k$, $a_{k+1}$, $b_{k+1}\in \ball(x^*,\epsilon)$ and $\ball(x^*,\epsilon)\subset U$, 
so that  we can apply the four-point estimate of part 1)   to the building block
$b_k\to a_{k+1}\to b_{k+1}$.  This gives
\begin{eqnarray}
\label{First}
d_B(a_{k+1})^2 + \ell \|b_k-b_{k+1}\|^2 \leq d_B(a_k)^2.
\end{eqnarray}
Now  using the cosine theorem and  (\ref{Angle}) we obtain
\begin{align*}
\|b_{k-1}-b_k\|^2 &= \|b_{k-1}-a_k\|^2+\|a_k-b_k\|^2 -2\|b_{k-1}-a_k\|\|a_k-b_k\|\cos \alpha_k\\
&= \left( \|b_{k-1}-a_k\|-\|a_k-b_k\|\right)^2+2 \|b_{k-1}-a_k\|\|a_k-b_k\|(1-\cos \alpha_k)\\
&\geq  \left(\|b_{k-1}-a_k\|-\|a_k-b_k\|\right)^2+2 \gamma \|b_{k-1}-a_k\|\|a_k-b_k\|^{\omega+1}\\
&\geq 2 \gamma\|b_{k-1}-a_k\|\|a_k-b_k\|^{\omega+1}.
\end{align*}
Since $b_k\in P_B(a_k)$ and $b_{k-1}\in B$, we have $\|b_{k-1}-a_k\|\geq \|a_k-b_k\|=d_B(a_k)$.
Hence $\|b_{k-1}-b_k\|^2 \geq 2 \gamma d_B(a_k)^{\omega +2}$, or what is the same
\begin{eqnarray}
\label{Second}
\|a_k-b_k\|^{-{(\omega+2)}/{2}}\|b_{k-1}-b_k\| \geq \sqrt{2\gamma}.
\end{eqnarray}
Recalling that $\theta = {(\omega+2)}/{4}$ we have $\theta \in[\frac{1}{2},1)$, because 
$\omega \in [0,2)$. By concavity of the function $s\mapsto s^{1-\theta}/(1-\theta)$
we have $s_1^{1-\theta}-s_2^{1-\theta} \geq (1-\theta) s_1^{-\theta} (s_1-s_2)$. Applying this
to $s_1 = d_B(a_k)^2$ and $s_2 = d_B(a_{k+1})^2$, we obtain
\begin{align*}
d_B(a_k)^{2(1-\theta)}-d_B(a_{k+1})^{2(1-\theta)}& \geq (1-\theta) d_B(a_k)^{-2\theta} \left( d_B(a_k)^2-d_B(a_{k+1})^2  \right)  \\
&= (1-\theta)  \|a_k-b_k\|^{-2\theta} \left( \|a_k-b_k\|^2 - \|a_{k+1}-b_{k+1}\|^2 \right)\\
&\geq (1-\theta)  \ell \sqrt{2\gamma}\, \frac{\|b_k-b_{k+1}\|^2}{\|b_{k-1}-b_k\|}, 
\end{align*}
where the last estimate  uses (\ref{First}) and (\ref{Second}).  Multiplying by
$\|b_k-b_{k-1}\|$ and recalling that $C=1/((1-\theta)\ell \sqrt{2\gamma})$,  this gives
\[
C\,\left( d_B(a_k)^{2(1-\theta)} - d_B(a_{k+1})^{2(1-\theta)} \right) \|b_k-b_{k-1}\| \geq \|b_k-b_{k+1}\|^2.
\]
By comparison of the arithmetic and geometric mean, 
${a^2} \leq bc$ implies $a\leq \frac{1}{2}b + \frac{1}{2}c$ for positive $a,b,c$, hence we obtain
\begin{eqnarray}
\label{sum}
\|b_k-b_{k+1}\| \leq \frac{1}{2}\|b_k-b_{k-1}\| + \frac{C}{2} \left( d_B(a_k)^{2(1-\theta)} - d_B(a_{k+1})^{2(1-\theta)} \right).
\end{eqnarray}
By the induction hypothesis we have (\ref{ind2}) at $k-1$, that is, 
\[
\sum_{{j=1}}^{k-1} \|b_j-b_{j+1}\| \leq \frac{1}{2} \sum_{{j=1}}^{k-1} \|b_{j-1}-b_j\| + \frac{C}{2} \left( d_B({a_{1}})^{2(1-\theta)}
-d_B(a_k)^{2(1-\theta)}\right).
\]
Adding this and (\ref{sum})  gives (\ref{ind2}) at index $k$. 

2.3) Let us now prove that the hypotheses (\ref{ind1}) and (\ref{ind2}) hold at ${k=1}$. 
Concerning (\ref{ind1}),
since ${b_{1}}\in \ball(x^*,4\delta)$ and $\|{b_{1}-a_{2}}\|\leq 4\delta\leq \frac{\epsilon}{4}$,
we have ${a_{2}}\in \ball(x^*,\frac{\epsilon}{2})$. Then using $\|{a_{2}-b_{2}}\|\leq \frac{\epsilon}{2}$ gives
${b_{2}}\in \ball(x^*,\epsilon)$, so (\ref{ind1}) is true at ${k=1}$. 

Concerning the validity of  (\ref{ind2}) at ${k=1}$, 
observe that using ${b_{0},b_{1},a_{1}}\in \ball(x^*,\frac{\epsilon}{4})$ we  may   repeat the
argument in the induction step starting before formula (\ref{First}) with ${k=1}$
in the place of $k$. 
The conclusion is formula (\ref{sum}) at ${k=1}$, that is, 
\[
\|{b_{1}-b_{2}}\| \leq \frac{1}{2} \|{b_{1}-b_{0}}\| + \frac{C}{2}
\left( d_B({a_{1}})^{2(1-\theta))} - d_B({a_{2}})^{2(1-\theta)}   \right),
\] 
and this is precisely (\ref{ind2}) at ${k=1}$. This concludes the induction argument.

3)
Having proved (\ref{ind1}), (\ref{ind2}) for all indices ${k \geq 1}$, 
we see from  (\ref{etape})
that the series $\sum_{{j=1}}^\infty \|b_j-b_{j+1}\|$ converges, 
which means $b_k$ is a Cauchy sequence, which converges to a limit $b^*\in B \cap \ball(x^*,\epsilon)$. 
Using relation (\ref{Second}) we conclude that
$a_k$ converges to the same limit $b^*\in A \cap B$.
\hfill $\square$
\end{proof}

Our next result gives the convergence rate for $\omega \in (0,2)$. The case $\omega=0$, where linear convergence
is obtained,
will be treated separately in Theorem \ref{theorem2}.

\begin{corollary}
\label{corollary}
{\bf (Rate of convergence)}.
Under the hypotheses of Theorem {\rm \ref{theorem1}}, {with $\omega\in (0,2)$,} the convergence rates 
are $\|b_k - b^*\| 
=\mathcal O\left(k^{-\frac{2-\omega}{2\omega}}  \right)$
and $\|a_k-b^*\| = \mathcal O \left( k^{- \frac{2-\omega}{2\omega}} \right)$.
\end{corollary}

\begin{proof}
Summing  (\ref{sum}) from $k=N$ to $k=M$
gives
\[
-\frac{1}{2} \|b_N-b_{N-1}\| + \frac{1}{2} \sum_{k=N}^{M-1} \|b_k-b_{k+1}\| + \|b_M-b_{M+1}\|
\leq \frac{C}{2} \left( d_B(a_N)^{2(1-\theta)} - d_B(a_{M+1})^{2(1-\theta)} \right).
\]
Passing to the limit $M\to\infty$ gives
\[
-\frac{1}{2} \|b_N-b_{N-1}\| + \frac{1}{2} \sum_{k=N}^\infty  \|b_k-b_{k+1}\| \leq  \frac{C}{2} d_B(a_N)^{2(1-\theta)}.
\]
Introducing $S_N=\sum_{k=N}^\infty \|b_k-b_{k+1}\|$, this becomes
\[
-\frac{1}{2} \left( S_{N-1}-S_N \right) + \frac{1}{2}S_N \leq \frac{C}{2} d_B(a_N)^{2(1-\theta)}.
\]
Consequently,
\[
\frac{1}{2}S_N \leq \frac{1}{2} (S_{N-1}-S_N) + {C}\,d_B(a_N)^{2(1-\theta)}.
\]
Now using estimate (\ref{Second}), we have
$d_B(a_N)^{2(1-\theta)} \leq (2\gamma)^{-\frac{1-\theta}{2\theta}} \|b_{N-1}-b_N\|^\frac{1-\theta}{\theta}$.
Putting $C' := C (2\gamma)^{-\frac{1-\theta}{2\theta}}$ and
substituting this gives
\begin{eqnarray}
\label{remarkable}
\frac{1}{2} S_N \leq \frac{1}{2}( S_{N-1}-S_N) + C' \left( S_{N-1}-S_N \right)^{\frac{1-\theta}{\theta}}.
\end{eqnarray}
Since $\theta \in (\frac{1}{2},1)$, we have $(1-\theta)/\theta \in (0,1)$. Moreover, $S_N\to 0$ implies $S_{N-1}-S_N\to 0$,
so the second term $(S_{N-1}-S_N)^{\frac{1-\theta}{\theta}}$ on the right of (\ref{remarkable}) dominates the first term $\frac{1}{2}(S_{N-1}-S_N)$.
That means, there exists another constant $C'' >0$ such that
\[
S_N^{\frac{\theta}{1-\theta}} \leq C'' (S_{N-1}-S_N)
\]
for all $N\in \mathbb N$.
We claim that there exists yet another constant $C'''$ such that
\begin{eqnarray}
\label{toshow}
1 \leq C'' (S_{N-1}-S_N) S_N^{-\frac{\theta}{1-\theta}}
\leq C''' \left(S_{N}^{-\frac{2\theta-1}{1-\theta}}-S_{N-1}^{-\frac{2\theta-1}{1-\theta}}  \right).
\end{eqnarray}
Assuming this proved, 
summation of (\ref{toshow}) from $N=1$ to $N=M$ gives
\[
M \leq C''' \left( S_{M}^{-\frac{2\theta-1}{1-\theta}}-S_1^{-\frac{2\theta-1}{1-\theta}}\right).
\]
Hence for yet two other constants $C'''',C'''''$,   $$S_M \leq C'''' \left[ S_1^{-\frac{2\theta-1}{1-\theta}} + M \right]^{-\frac{1-\theta}{2\theta-1}}
\leq C''''' M^{-\frac{1-\theta}{2\theta-1}}.$$ Since $\|b_M-b^*\| \leq S_M$ by the triangle inequality,
that proves the claimed speed of convergence.

In order to prove (\ref{toshow}) we divide the set of indices
into $\mathcal I = \{N: 2S_N \geq  S_{N-1}\}$ and $\mathcal J=\{N: 2S_N < S_{N-1}\}$.  For $N\in \mathcal I$ we have
\begin{align*}
(S_{N-1}-S_N) S_N^{-\frac{\theta}{1-\theta}}
&\leq 2^{\frac{\theta}{1-\theta}} (S_{N-1} - S_N) S_{N-1}^{-\frac{\theta}{1-\theta}}\\
&\leq 2^{\frac{\theta}{1-\theta}}
\int_{S_N}^{S_{N-1}} S^{-\frac{\theta}{1-\theta}}\, dS \\
&= 2^{\frac{\theta}{1-\theta}}
\frac{1-\theta}{2\theta-1} \left( S_N^{-\frac{2\theta-1}{1-\theta}} - S_{N-1}^{-\frac{2\theta-1}{1-\theta}}\right),
\end{align*}
proving (\ref{toshow}).  
In contrast, for $N\in \mathcal J$ we have
\[
S_N^{-\frac{2\theta-1}{1-\theta}} 
-S_{N-1}^{-\frac{2\theta-1}{1-\theta}} 
\geq 2^{\frac{2\theta-1}{1-\theta}} S_{N-1}^{-\frac{2\theta-1}{1-\theta}}-S_{N-1}^{-\frac{2\theta-1}{1-\theta}}
=\left(2^{\frac{2\theta-1}{1-\theta}}-1\right)S_{N-1}^{-\frac{2\theta-1}{1-\theta}} \to \infty
 \]
in view of $S_{N-1}\to 0$, $\frac{2\theta-1}{1-\theta}>0$ and $2^{\frac{2\theta-1}{1-\theta}}-1 >0$.  So on the  set
$\mathcal J$ estimate (\ref{toshow})
is also satisfied. Finally, the same estimate for $a_k$ follows from
$\|a_{k+1}-b^*\| \leq \|a_{k+1}-b_k\| + \|b_k-b^*\| \leq 2 \|b_k-b^*\|$.
\hfill $\square$
\end{proof}

\begin{theorem}
\label{linear}
{\bf (Local convergence with linear rate)}.
Let
$A,B$ intersect $0$-separably at $x^*$ with 
constant $\gamma\in (0,2)$. Suppose  $B$ is $0$-H\"older regular at $x^*$ with respect to $A$
with constant $c < \frac{\gamma}{2}$. Then there exists a neighborhood $V$
of $x^*$ such that every sequence of alternating projections
that enters $V$ converges  R-linearly to a point $b^*\in A\cap B$. 
\end{theorem}

\begin{proof}
Applying Lemma \ref{third} and  Theorem \ref{theorem1}
in the case $\omega=0$, we obtain convergence of the sequence
$a_k,b_k$ to a point $b^*\in A \cap B$ from summability of $\sum_k \|b_{k-1}-b_k\|$. 
Now from the proof of Corollary \ref{corollary}, we see that in the case $\theta=\frac{1}{2}$ 
equation (\ref{remarkable}) simplifies to
\[
\frac{1}{2} S_N\leq \frac{1}{2}(S_{N-1}-S_N) + C' (S_{N-1}-S_N),
\]
or what is the same
\[
S_N \leq \frac{1+2C'}{2+2C'}\,S_{N-1}.
\]
This proves Q-linear convergence of $S_N$ to $0$, hence R-linear convergence of $b_N\to b^*$.
\hfill $\square$
\end{proof}

\begin{remark}
Theorem \ref{linear} extends the results in \cite[Thm. 5.2]{luke}
and \cite[Thm. 3.14]{bauschke1} in two ways. 
Firstly, as seen in example \ref{packman}, $0$-H\"older regularity includes sets $B$ which have  
singularities pointing inwards, where superregularity  \cite{luke}    and
its extension in \cite{bauschke1} fail.  Secondly, $0$-separability is weaker than
linear regularity or the CQ in \cite{bauschke1}, see example \ref{spiral2}.
\end{remark}

We now obtain the following consequence of Theorem \ref{theorem1}, originally
proved in \cite{bauschke2} for more general families of sets. When specialized to the case of two sets
we have

\begin{corollary}
{\bf (Bauschke et al. \cite[{Theorem 3.14}]{bauschke2})}.  {Suppose $(A,\tilde{A},B,\tilde{B})$ satisfies the {\rm CQ}-condition {\rm (\ref{CQ:condition})}
at $x^*\in A \cap B$, where $P_A(\partial B\setminus A)\subset \tilde{A}$, $P_B(\partial A\setminus B)\subset\tilde{B}$}.
Moreover, suppose for every $\epsilon > 0$ there exists $\delta > 0$ such that
$B$ is $(A,\epsilon,\delta)$ regular at $x^*$. Then there exists a neighborhood $V$
of $x^*$ such that every alternating sequence which enters $V$
converges R-linearly {to a point in $A\cap B$}. \hfill $\square$
\end{corollary}

As already shown in \cite{bauschke2} one readily derives

\begin{corollary}
{\bf (Lewis, Luke, Malick \cite{luke})}. Suppose $A,B$ intersect linearly regularly 
and $B$ is superregular. Then alternating projections 
converge locally R-linearly to a point in the intersection. \hfill $\square$
\end{corollary}

The following is now a consequence of Theorem \ref{linear}, using Propositions \ref{hoax} and \ref{kill}.

\begin{corollary}
{\bf (Drusvyatskiy, Ioffe, Lewis {\rm \cite{obsolete}})}.
Suppose $A,B$ intersect intrinsically transversally at $x^*$. Then there exists a neighborhood
$U$ of $x^*$ such that every sequence of alternating projections
entering $U$ converges R-linearly to a point in the intersection.
\end{corollary}

\begin{proof}
By Proposition \ref{hoax} the sets
$A,B$ intersect $0$-separably at $x^*$ with constant $\gamma=\kappa^2/2\in (0,1]$.
By Proposition \ref{kill},   $B$ is $0$-H\"older regular
with respect to $A$ at $x^*$ with any constant $c < \kappa^2/(1-\kappa^2)=\gamma/(2-\gamma)$. 
Choosing  $c<\gamma/2$ allows us therefore to apply Theorem \ref{linear}.
\hfill $\square$
\end{proof}

\begin{remark}
Drusvyatskiy {\em et al.}  \cite{obsolete} stress that their approach gives local R-linear convergence
under a   transversality hypothesis alone, while the older \cite{luke,bauschke2,initial} still need 
regularity assumptions. However, this statement should be read with care, 
because Propositions \ref{hoax} and \ref{kill}  show that
intrinsic transversality amalgamates transversality and regularity aspects. In particular,
it is more restrictive than $0$-H\"older regularity in tandem with $0$-separability, 
so that Theorem \ref{linear}
is stronger than the main result in \cite{obsolete}. 
\end{remark}

\section{Subanalytic sets}
\label{loja}
Following \cite{bierstone}, a subset $A$ of $\mathbb R^n$ is called {\em semianalytic}
if for every $x\in \mathbb R^n$ there exists an open neighborhood $V$ of $x$ such that
\begin{eqnarray}
\label{semi}
A \cap V = \bigcup_{i\in I} \bigcap_{j\in J} \{x\in V: \phi_{ij}(x)=0, \psi_{ij}(x) > 0\}
\end{eqnarray}
for finite sets $I,J$ and real-analytic functions $\phi_{ij},\psi_{ij} :V \to \mathbb R$. The set $B$ in $\mathbb R^n$
is called {\em subanalytic} if for every $x\in \mathbb R^n$ there exist a neighborhood $V$
of $x$ and a bounded semianalytic subset $A$ of some $\mathbb R^n \times \mathbb R^m$,
$m\geq 1$, such that $B\cap V =\{x\in \mathbb R^n: \exists y\in \mathbb R^m \mbox{ such that }(x,y)\in A\}$. Finally,
an extended real-valued function $f:\mathbb R^n \to \mathbb R \cup\{\infty\}$ is called subanalytic if its
graph is a subanalytic subset of $\mathbb R^n\times \mathbb R$.

We consider the function $f:\mathbb R^n \to \mathbb R \cup \{\infty\}$
defined as $f(x) = i_A(x) + \frac{1}{2} d_B(x)^2$, where $i_A$ is the indicator
function of $A$; cf. \cite[p. 84]{mord}. 

\begin{lemma}
\label{first}
Let $f = i_A + \frac{1}{2}d_B^2$.
Let $a^+\in A$ be projected from $b\in B$ and $v= \lambda(b-a^+)\in {N^p_A(a^+)}$, where
$\lambda \geq 0$.
Then $v + a^+ - P_B(a^+) \subset \widehat{\partial} f(a^+)$.
\end{lemma}

\begin{proof}
Note that $a^+- P_B(a^+)\subset \widehat{\partial} \left(\frac{1}{2} d_B^2\right)(a^+)$ by \cite[1.3.3]{mord} or \cite[Example 8.53]{rock}. 
Next observe that $\widehat{\partial} i_A(a^+) = \widehat{N}_A(a^+)$ by \cite[Proposition 1.79]{mord} or  \cite[Exercice 8.14]{rock}. Hence $v \in N_A^p(a^+) \subset \widehat{N}_A(a^+) = \widehat{\partial}i_A(a^+)$  using \cite[Lemma 2.4]{bauschke1}. Finally,
by the sum rule \cite[Prop. 1.12]{schirotzek}
we have $\widehat\partial i_A(a^+)+\widehat\partial  \frac{1}{2}d_B^2(a^+) \subset \widehat\partial f(a^+)$, 
which completes the proof.
\hfill $\square$
 \end{proof}

\begin{definition}
\label{def_loja}
Let $f:\mathbb R^n \to \mathbb R \cup\{\infty\}$ be lower semi-continuous with closed domain
such that $f|_{{\rm dom} f}$ is continuous. We say that $f$ satisfies the \L ojasiewicz inequality
with exponent $\theta \in [0,1)$ at the critical point
$x^*$ of $f$ if there exists $\gamma > 0$ and a neighborhood $U$ of $x^*$ such that
$|f(x)-f(x^*)|^{-\theta} \|g\| \geq \gamma$ for every $x\in U$ and every $g\in \widehat{\partial} f(x)$.
\hfill $\square$
\end{definition}

Here $x^*$ is critical if $0\in \partial f(x^*)$, see \cite{mord,rock}. 

\begin{lemma}
\label{second}
Suppose $f = i_A+\frac{1}{2}d_B^2$ satisfies the \L ojasiewicz
inequality with exponent $\theta \in [0,1)$ at $a^*=b^*\in A \cap B$ and constant $\gamma >0$. Then in fact $\theta > \frac{1}{2}$. 
Moreover,  \textcolor{black}{$B$ intersects $A$ separably}  with exponent
$\omega =4\theta -2\in (0,2)$ and constant $\gamma'=2^{-2\theta-1}\gamma^2$.
\end{lemma}

\begin{proof}
Note that $f(a^*)=0$.
Therefore there exists a neighborhood $U$ of  $a^*\in A \cap B$ such that
\begin{eqnarray}
\label{prev}
f(a^+)^{-\theta} \|g\| \geq \gamma >0
\end{eqnarray}
for every $a^+\in A\cap U$ and every $g\in \widehat{\partial} f(a^+)$. Now let $a\to b\to a^+\to b^+$
be a building block with $a,b,a^+,b^+\in U$. 
From Lemma \ref{first}, $g = v + a^+-b^+\in \widehat{\partial} f(a^+)$
for every $v\in N_A^p(a^+)$ of the form $v=\lambda(b-a^+)$.  This uses the fact  that $a^+\in P_A(b)$. 
Hence by  (\ref{prev}) we have
 \[
 2^\theta 
 d_B(a^+)^{-2\theta} \|\lambda (b-a^+) + a^+ - b^+\| \geq \gamma> 0
 \]
for every $\lambda \geq 0$. We deduce
\begin{eqnarray}
\label{min}
d_B(a^+)^{-2\theta} \min_{\lambda\geq 0} \|\lambda(b-a^+) + a^+ - b^+\| \geq 2^{-\theta}\gamma.
\end{eqnarray}
Let us  for the time being consider angles $\alpha= \angle(b-a^+,b^+-a^+)$ smaller than $90^\circ$.
Then the minimum value in (\ref{min}) is
$d_B(a^+)^{-2\theta}\|a^+-b^+\| \sin\alpha$. Therefore 
\begin{eqnarray}
\label{erg}
\frac{\sin\alpha}{d_B(a^+)^{2\theta-1}} \geq  2^{-\theta} \gamma.
\end{eqnarray}
Since $1-\cos\alpha\geq\frac{1}{2}\sin^2\alpha$,  estimate (\ref{erg})
implies
\begin{eqnarray}
\label{almost}
\frac{1-\cos\alpha}{d_B(a^+)^{4\theta-2}} \geq  2^{-2\theta-1}\gamma^2.
\end{eqnarray}
This shows that we must have $\theta > \frac{1}{2}$,
because the numerator tends to 0,
so the denominator  has to go to zero, too, which it does for $4\theta -2>0$.  

Let us now discuss the case  where $\alpha \geq 90^0$. We claim that the same estimate (\ref{almost})
is still satisfied. 
Since $\cos\alpha < 0$,
the numerator $1-\cos\alpha$ in (\ref{almost}) is $\geq 1$. Moreover, the infimum in (\ref{min}) is now attained
at $\lambda=0$ with the value $\|a^+-b^+\|=d_B(a^+)$. Hence estimate (\ref{min})
implies $d_B(a^+)^{1-2\theta} \geq 2^{-\theta}\gamma$, hence
$d_B(a^+)^{2-4\theta} \geq 2^{-2\theta}\gamma^2> 2^{-2\theta-1}\gamma^2$, so  that (\ref{almost})
is satisfied.
This completes the proof.
\hfill $\square$
\end{proof}

\begin{theorem}
\label{theorem2}
Let $A,B$ be
closed subanalytic sets. Then $A,B$ intersect separably.  
\end{theorem}

\begin{proof}
We assume $A\cap B\not=\emptyset$, otherwise there is nothing to prove.
Consider the function $f:\mathbb R^n \to \mathbb R \cup\{\infty\}$
defined as $f(x) = i_A(x) + \frac{1}{2}d_B(x)^2$. Then $f$ has closed domain $A$ and is continuous on 
$A$, which makes it amenable to definition \ref{def_loja}. Every
$x^*\in A \cap B$ is a critical point of $f$. Since $A,B$ are subanalytic sets,
$f$ is subanalytic. That can be seen as follows. First observe that
$d_B$ is subanalytic by \cite[I.2.1.11]{shiota}. Then $d_B^2$ is subanalytic
as the product of two subanalytic functions by \cite[I.2.1.9]{shiota}. Finally, graph$(f) = (A \times \mathbb R) \cap {\rm graph}(\frac{1}{2}d_B^2)$
shows that $f$ is subanalytic.

Now we invoke Theorem 3.1 of \cite{bolte1},
which asserts that $f$ satisfies the \L ojasiewicz inequality at $x^*$ for some $\theta \in (0,1)$. 
Hence (\ref{prev}) is true for every $g\in \partial f(a^+)$, and therefore also for every $g\in \widehat{\partial}f(a^+)$.
Applying Lemma
\ref{second}, we deduce that $B$ intersects $A$ separably with $\omega=4\theta-2$. Interchanging
the roles of $A$ and $B$, it follows also that $A$ intersects $B$ separably.
\hfill $\square$
\end{proof}

\begin{corollary}
\label{cor5}
{\bf (Local convergence for subanalytic sets)}.
Let $A,B$  be subanalytic. Suppose 
$B$ is H\"older regular at $x^*\in A \cap B$ with respect to $A$.
Then there exists a neighborhood $V$ of $x^*$ such that every sequence of alternating projections $a_k,b_k$
which enters $V$ converges to some $b^*\in A \cap B$ with rate
$\|b_k-b^*\|=\mathcal O(k^{-\rho})$ for some $\rho\in (0,\infty)$.
\hfill $\square$
\end{corollary}

\begin{corollary}
\label{mainsub}
Let $A,B$ be closed subanalytic sets and suppose $B$
has slowly vanishing reach with respect to $A$. Let $x^*\in A \cap B$, then
there exists a neighborhood $U$ of $x^*$ such that every sequence
of alternating projections $a_k,b_k$ which enters $U$ converges to some $b^*\in A \cap B$ with  rate $\|b_k-b^*\| = \mathcal O(k^{-\rho})$ for
some $\rho \in (0,\infty)$. \hfill $\square$
\end{corollary}

Recall from \cite{bierstone,shiota} that a subset $A$ of $\mathbb R^n$ is called
{\em semialgebraic} if for every $x\in\mathbb R^n$ there exists a neighborhood
$V$ of $x$ such that (\ref{semi}) is satisfied with  $\phi_{ij},\psi_{ij}$ polynomials. Naturally, this means that every semialgebraic set is
semianalytic, hence subanalytic. By combining Theorems  \ref{theorem1}
and \ref{theorem2}, we therefore obtain the following result.

\begin{corollary}
{\bf (Borwein, Li,  Yao \cite{borwein})}. 
Let $A,B$ be closed convex semialgebraic sets with nonempty intersection. Then there exists
$\rho\in (0,\infty)$ such that every sequence
of alternating projections converges with rate $\|b_k-b^*\| = \mathcal O(k^{-\rho})$.
\hfill $\square$
\end{corollary}

As a variant of the method of alternating projects
consider the averaged projection method. Given closed sets $C_1,\dots,C_m$,
the method generates
a sequence $x_n$ by the recursion $x_{n+1}\in \frac{1}{m}\left(P_{C_1}(x_n)+\dots +P_{C_m}(x_n)\right)$.

\begin{corollary}
Let $C_1,\dots,C_m$ be subanalytic sets in $\mathbb R^d$, and let $c^*\in C_1 \cap \dots \cap C_m$. Then there exists a neighborhood
$U$ of $c^*$ such that whenever a sequence $x_n$ of averaged projections 
enters $U$, then it converges to some $x^*\in C_1 \cap \dots \cap  C_m$ with rate ${\|x_k-x^*\|}=\mathcal O(k^{-\rho})$
for some $\rho\in (0,\infty)$.
\end{corollary}

\begin{proof}
We follow a standard procedure and define
closed sets in the product space $\mathbb R^d \times \dots \times \mathbb R^d$ ($m$ times)
as $B = \{(x,\dots,x): x\in \mathbb R^d\}$, and
$A=C_1\times \dots \times C_m$. Note that $A$ is again
subanalytic by \cite[I.2.1.1]{shiota},  whereas $B$ is  convex and subanalytic. We have
$P_B(x_1,\dots,x_m)= (\frac{1}{m}(x_1+\dots +x_m),\dots, \frac{1}{m}(x_1+\dots+x_m))$, while $P_A (x_1,\dots,x_m)
=(P_{C_1}(x_1),\dots,P_{C_m}(x_m))$. Therefore a sequence of averaged projections between
$C_1,\dots,C_m$  generates a sequence of alternating projections between $A,B$.

Since $B$ is convex, it is prox-regular hence H\"older regular with respect to $A$, so
by Corollary \ref{cor5} there exists a neighborhood
$\mathcal U=U\times\dots\times  U$ of $(c^*,\dots,c^*)\in A \cap B$
such that every alternating sequence which enters  $\mathcal U$ converges to some  $(x^*,\dots,x^*)\in A\cap B$
with rate $\mathcal O(k^{-\rho})$ for some $\rho\in (0,\infty)$.
Now consider an averaged projection sequence $x_k$ entering $U$. It follows
that $(x_k,\dots,x_k)\in \mathcal U$, hence $x_k$  converges to $x^*$ with that same rate.
\hfill $\square$
\end{proof}

\begin{remark}
We mention a related averaged projection method
in \cite[Corollary 12]{attouch}, where the authors use the Kurdyka-\L ojasiewicz inequality. The employed technique indicates that
results in the spirit of Theorem \ref{theorem2} could be obtained for more general classes of sets definable in an o-minimal structure \cite{attouch}.
\end{remark}

We conclude this section with an application of
Theorem \ref{theorem1},  demonstrating its versatility as a convergence test
in practical situations. 
Let
$\mathbb C^N$ be a finite dimensional unitary space,  
and  consider the discrete Fourier transform 
\[
\widehat{x}(\omega)= \frac{1}{\sqrt{N}} \sum_{t=0}^{N-1} e^{2\pi i t\cdot \omega /N} x(t), \quad \omega=0,\dots,N-1
\]
as a unitary linear operator  $x \to \widehat{x}$ of $\mathbb C^N$. The phase retrieval problem \cite{gerchberg,elser}
consist in estimating an unknown signal $x\in \mathbb C^N$ whose Fourier
amplitude $| \widehat{x}(\omega)|=a(\omega)$, $\omega=0,\dots,N-1$, is known.  In physical terminology, 
identifying $x$  means retrieving
its unknown phase $\widehat{x}(\omega)/|\widehat{x}(\omega)|$ in frequency domain. 

Formally, given
a function $a(\cdot):\{0,\dots,N-1\} \to [0,\infty)$, we have to find an element of the set
$B=\{x\in \mathbb C^N: |\widehat{x}(\omega)| = a(\omega) \mbox{ for all } \omega=0,\dots,N-1\}$.  
Since this problem
is underdetermined,  additional information about $x$
in a different Fourier plane or  in the time domain
is added.  We represent it  in the abstract form $x\in A$ for a closed set $A$.
Then the phase retrieval problem is to find $x\in A \cap B$.

The famous Gerchberg-Saxton error reduction
algorithm \cite{gerchberg} computes a solution of the phase retrieval
problem by generating a sequence of estimates as follows:
Given $x\in \mathbb C^N$,
compute $\widehat{x}$ and correct its Fourier amplitude by putting
${y}(\omega) = a(\omega)\, \widehat{x}(\omega)/|\widehat{x}(\omega)|$ if $\widehat{x}(\omega)\not=0$,
and ${y}(\omega)=a(\omega)$ if there is extinction $\widehat{x}(\omega)=0$. For short, ${y}=a\widehat{x}/|\widehat{x}|$
with the convention $0/|0|=1$.
Then
compute the inverse discrete Fourier transform $\widetilde{{y}}$  of $y$ and 
build the new iterate $x^+$ by projecting $\widetilde{y}$ on the set $A$, that is $x^+\in P_A(\widetilde{y})$.
In condensed notation:
\begin{eqnarray}
\label{GS}
x^+ \in P_A\left( (a\widehat{x}/|\widehat{x}|)^\sim \right).
\end{eqnarray}

\begin{corollary}
\label{gerchberg}
 {\bf (Gerchberg-Saxton error reduction).}
Suppose the  constraint $x\in A$ is represented by a subanalytic set $A$. Let
$x^*\in A \cap B$ be a solution of the phase retrieval problem.
Then there exists $\epsilon > 0$ such that whenever a Gerchberg-Saxton sequence
$x_k$  enters $\ball(x^*,\epsilon)$, then it converges to a solution $\bar{x}\in A \cap B$ of the 
phase retrieval problem with rate of convergence  $\|x_k-\bar{x}\|=\mathcal O(k^{-\rho})$
for some $\rho\in (0,\infty)$.
\end{corollary}

\begin{proof}
With the convention $0/|0|=1$,  the mapping
$x \mapsto \left( a\, \widehat{x}/|\widehat{x}| \right)^\sim$ is an orthogonal projection on the set
$B=\{x\in \mathbb C^N: |\widehat{x}(\cdot)|=a(\cdot)\}$. (See for instance \cite[(8),(10)]{BCL}, where the authors consider even the function space case).
Therefore the Gerchberg-Saxton algorithm (\ref{GS}) is an instance of the alternating projection
methods between the subanalytic set $A$ and the Fourier amplitude  set $B$. We show that
$B$ is subanalytic  and  prox-regular.  Local convergence with rate $\mathcal O(k^{-\rho})$ then follows
from Corollary \ref{mainsub}.  

As far as subanalyticity of $B$ is concerned, observe that on identifying $\mathbb C^N$
with $\mathbb R^{2N}$ via $\widehat{x}(\omega) = \widehat{x}_1(\omega) + i\, \widehat{x}_2(\omega)$, we have
\[
B=\bigcap_{\omega=0}^{N-1}\left\{x\in \mathbb C^N: \widehat{x}_1(\omega)^2 + \widehat{x}_2(\omega)^2 - a(\omega)^2=0\right\} ,
\]
which is clearly a representation of the form (\ref{semi}), since the discrete Fourier transform $x \mapsto \widehat{x}$ is analytic.

To show prox-regularity of $B$, we have to show that the projection on $B$ is single-valued in a neighborhood of $B$. 
With the same identification $\mathbb C^N \cong \mathbb R^{2N}$ evoked before, 
the projection on $B$ splits into $N$ projections in $\mathbb R^2$, 
given as $(\widehat{x}_1(\omega),\widehat{x}_2(\omega))\to a(\omega) \left( \frac{\widehat{x}_1(\omega)}{\sqrt{\widehat{x}_1(\omega)^2+\widehat{x}_2(\omega)^2}},  
\frac{\widehat{x}_2(\omega)}{\sqrt{\widehat{x}_1(\omega)^2+\widehat{x}_2(\omega)^2}  }\right)$.
In the case $a(\omega) = 0$ this is the projection onto the origin, which is clearly single-valued.
For $a(\omega) > 0$ this is the orthogonal projection 
onto the sphere of radius $a(\omega)$ in $\mathbb R^2$, which is single-valued except at the origin $(\widehat{x}_1(\omega),\widehat{x}_2(\omega))=(0,0)$. 
This means the projection on $B$ is unique on the neighborhood
$U=\{x\in \mathbb C^N: |\widehat{x}(\cdot)| \geq a(\cdot)/2\}$ of $B$, proving prox-regularity of $B$.
\hfill $\square$
\end{proof}

\begin{remark}
The constraint $x\in A$ may represent additional measurements, or it may
include prior information about the unknown image. In the original work \cite{gerchberg} 
$x\in A$  represents Fourier amplitude information from a second
Fourier plane, which is a constraint analogous to $x\in B$. The constraint $x\in A$ may also represent prior information 
about the support supp$(x)$ of the unknown signal $x$ in physical domain. It may for instance
be known that supp$(x)\subset S$, where $S$ is a subset of
$\{0,\dots,N-1\}$ with card$(S)\ll N$, or with a periodic structure. This is  known as an atomicity constraint in crystallographic
phase retrieval \cite{elser}. For $A=\{x\in \mathbb C^N: x(t)=0 \mbox{ for } t\not\in S\}$, 
$P_A$ is simply truncation $y \to y \cdot 1_S$.  Here
the Gerchberg-Saxton error correction method has the explicit form
\[
x^+(t) = \left\{
\begin{array}{ll}
\left( a\widehat{x}/|\widehat{x}| \right)^\sim(t) &\mbox{if } t\in S\\
0&\mbox{else}
\end{array}
\right.
\]
Other choices of the constraint $x\in A$ have been discussed in the literature, see e.g. \cite{elser}.
Our convergence result requires only subanalyticity of $A$, a condition which is always satisfied in practice. 
\end{remark}

\section{Examples}
\label{examples}

\begin{example}
\label{packman}
{\bf (Packman gulping an ice-cream cone  the wrong way).}
Consider packman
$B=\{(x,y)\in \mathbb R^2: x^2+y^2 \leq 1, x \leq |y|\}$ the instant before it
scarfs down
the ice-cream cone section $A=\{(x,y) \in \mathbb R^2: 0\leq x^2+y^2 \leq 1, 2|y| \leq x\}$
fitting symmetrically into
its notch. We have $A \cap B = \{(0,0)\}$, leaving an angular gap 
of $15^\circ$ on both sides. 

Let $a^+=(x,\frac{1}{2}x)\in \partial A$, then $b^+=P_B(a^+)=(\frac{3}{4}x,\frac{3}{4}x)\in \partial B$,
which means $r=\|b^+-a^+\|=\frac{\sqrt{2}}{4}x$. It is easy to see that condition (\ref{empty})
is satisfied for every $c<1$ and arbitrary $\sigma\in [0,1)$, i.e., $B$ is H\"older regular with respect to $A$.
This example shows that H\"older regularity applies to sets which have inward corners
and fail Clarke regularity.

Note that since $B$ is 
not Clarke regular at $x^*=(0,0)$, 
it is not superregular  in the sense of \cite{luke}.
What is more, $B$ is not $(A,\epsilon,\delta)$-regular
in the sense of \cite{bauschke1}
at $x^*=(0,0)$,  regardless how $\epsilon,\delta > 0$ are chosen,  because the cone 
$b^++\{v: \langle a^+-b^+,v\rangle \leq \epsilon \|a^+-b^+\|\|v\|\}$ with vertex 
at the projected point $b^+=(\frac{3}{4}x,\frac{3}{4}x)\in B$
hits $B$ at points $b'\in B$ other than $b$ on the opposite side of $A$, regardless how small 
$\epsilon$ is chosen. And this cannot be prevented  by shrinking the neighborhood $\ball(x^*,\delta)$. 

Note that $A,B$ intersect $0$-separably at $(0,0)$, hence alternating projections converge linearly
by Theorem \ref{linear}. This cannot be obtained from the results in \cite{luke,bauschke1}.
\hfill $\square$
\end{example}

\begin{example}
\label{example2}
{\bf (Regularity cannot be dispensed with)}.
Following \cite{bauschke-noll},
consider the spiral $z(\phi)=(1+e^{-\phi})e^{i\phi}$, $\phi\in [0, \infty)$ in the complex plane 
which approaches the unit circle $S=\{|z|=1\}$
form outside. Define a sequence $z_n = z(\phi_n)$ 
with $\phi_1 < \phi_2 < \dots\to\infty$ such that $\|z_{n+1}-z_n\| < \|z_n-z_{n-1}\|\to 0$,
$P_{\{z_k:k\not=n\} \cup S}(z_n)=z_{n+1}$, and  such that every $z\in S$ is an accumulation point of the 
$z_n$.
In \cite{bauschke-noll} an explicit construction with these properties is obtained recursively
as
\begin{eqnarray}
\label{constr}
z_n = z(\phi_n),\quad 
\|z(\phi_{n+1})-z(\phi_n)\|=r_n, \quad
r_{n+1}=e^{-\phi_{n+1} } \frac{1-e^{-2\pi}}{2}.
\end{eqnarray}
 Let $A = \{z_{2n}:n\in \mathbb N\}\cup S$, $B=\{z_{2n-1}: n\in \mathbb N\} \cup S$, then
$A \cap B=S$. Note that for starting points $|z_0|>1$, the sequence of alternating projections between
$A$ and $B$ is a tail of the sequence $z_n$, so none of the alternating sequences converges. 
Note that $\angle(z_{n+1}-z_n,z_{n-1}-z_n) \to \pi$,
hence $A,B$ intersect $0$-separably at every $x^*\in S=A\cap B$. 
The CQ
in the sense of \cite{bauschke1} is satisfied  at every $x^*\in A \cap B$. Namely,
for $z\in S$, $N_A^B(z)=N_B^A(z)=\mathbb R_+(-i{z})$. Indeed, as $a_n= P_A(b_n)$ approaches
$z$, the direction $u_n= (b_n-a_n)/\|b_n-a_n\|$ approaches a direction perpendicular to $z$, and since the spiral turns counterclockwise,
this direction is  $-i{z}$. Therefore $N_A^B(z) \cap (-N_B^A(z))=\{0\}$ for every $z\in S$.

Since the sequence $z_n$ fails to converge,
we conclude that this must be  due to the lack of regularity at points in $S$. Indeed,
H\"older regularity fails for every $0\leq \sigma < 1$. 
This can be seen as follows. Since the angle $\angle(b-a^+,b^+-a^+)$ for the building
block $b \to a^+ \to b^+$ approaches $\pi$,
the corresponding angle $\beta = \angle(b-b^+,a^+-b^+)$
goes to 0, so $\cos\beta \to 1$, and for $\sigma \in (0,1)$
we cannot find $c>0$ such that $\cos\beta \leq \sqrt{c}r^\sigma$. Therefore, in order to assure
(\ref{empty}), we would need $b\not\in \ball(a^+,(1+c)r)$, where $r=\|a^+-b^+\|$. This however, would imply
linear convergence of the alternating sequence, which fails. As a consequence of
Proposition \ref{regular}, the other regularity concepts fail, {as does intrinsic transversality}.
\hfill $\square$
\end{example}

\begin{example}
\label{example3}
{\bf (Discrete spiral I).}
We consider a discrete approximation of the logarithmic spiral, generated by  $8$ equally spaced rays emanating
from the origin. Starting on one of the rays, we project perpendicularly on the 
neighboring ray, going counterclockwise. We label the projected points  $a_1,b_1,a_2,\dots$.
This defines two sets $A=\{a_i: i\in \mathbb N\}\cup \{(0,0)\}$ and 
$B=\{b_i:i\in \mathbb N\} \cup\{(0,0)\}$ with $A\cap B=\{(0,0)\}$ such
that $P_B(a_i)=b_i$ and $P_A(b_i)=a_{i+1}$. Every sequence of alternating projections between $A$ and $B$ 
not starting at the origin is a tail of the sequence $a_n,b_n$ and converges to $(0,0)$.

Since  $\alpha = \angle(b-a^+,b^+-a^+)=135^\circ$,  $A,B$ intersect  $0$-separably
at $x^*=(0,0)$.  We  check whether the intersection satisfies the CQ in the sense on \cite{bauschke1}. 
Consider one of the rays on which a point $a^+$ is situated. Then $u=b-a^+ \in \widehat{N}_A^B(a^+)$
is perpendicular to $a^+-x^*$, i.e., perpendicular to the ray in question. As $u$ is the same for all $a^+$
on that ray, we have $u\in N_A^B(0,0)$. Altogether, $N_A^B(0,0)={\rm lin}\{u_1,u_3,-u_1,-u_3\}$   for four directions
spaced $90^\circ$. Similarly, $N_B^A(0,0)=\{u_2,u_4,-u_2,-u_4\}$  spaced $90^\circ$, and intertwined 
with the directions of $N_A^B(0,0)$. We have $N_A^B(0,0)=-N_A^B(0,0)$, and similarly
for $N_B^A(0,0)$, and since $N_B^A(0,0)\cap N_A^B(0,0)=\{(0,0)\}$, the intersection does indeed satisfy the CQ in the sense
of \cite{bauschke1} for $\widetilde{A}=A$, $\widetilde{B}=B$.

How about regularity at $(0,0)$?
 Naturally, $A,B$ are not superregular
at $(0,0)$, because they are not  Clarke regular.
Concerning $(A,\epsilon,\delta)$-regularity of $B$ in the sense of \cite{bauschke1},  suppose in a building block
$b\to a^+\to b^+$ we wish to set up a cone with apex $b^+$
and axis $b^++\mathbb R_+(a^+-b^+)$ by  choosing its aperture small enough through the choice of $\epsilon$ such that all previous
points of $A$ are avoided, then we have to choose smaller and smaller angles $\beta$ to do this,
so  this type of regularity fails.

On the other hand, we have $\sigma$-H\"older regularity for every
$\sigma\in [0,1)$. Suppose we start
at $a_1=(1,0)$, then $b_1=(\frac{1}{2},\frac{1}{2})$ and $a_2=(0,\frac{1}{2})$. After a tour of $360^\circ$,
the spiral comes back to the horizontal ray $\mathbb R^+(1,0)$ at $a_5=(\frac{1}{16},0)$. So while at the beginning the
spiral turns within the square $[-1,1]^2$, from the second tour onward it will stay
in the square $[-\frac{1}{16},\frac{1}{16}]^2$. As the circle $\ball(b_1,\frac{7}{16}\sqrt{2})$
contains no points of the small square in its interior,  the distance of $b_1$ to the small square
being $R=\frac{7}{16}\sqrt{2}$, writing $R=(1+c)r$ we conclude that we can take
$c=\frac{7}{8}\sqrt{2}-1>0$ in (\ref{empty}). Now up to a scaling and a rotation the situation is precisely the same 
for {\em every} building block $a\to b \to a^+$ starting in a square of length $2 \|a\|$. After one $360^\circ$-tour
we end up at $a^{++++}$ on the same ray as $a$, and from there on
the spiral will stay in that smaller square of length $2 \frac{1}{16}\|a\| = 2\|a^{++++}\|$.
As a consequence of theorem \ref{linear}, the sequence
converges to $(0,0)$ with linear rate. None of the approaches
of \cite{luke,malick,bauschke1} allows to derive this.
\hfill $\square$
\end{example}

\begin{example}
\label{spiral2}
{\bf (Discrete spiral II)}. We can modify the above construction by fixing  $\phi \in (0, \frac{\pi}{4})$ and generating rays
$k\phi$, $k\in \mathbb N$. Turning counterclockwise, and keeping only the projected
points, we generate iterates $a_k,b_k$ with the property that 
$a_k$ has angle $2k\phi$ mod $2\pi$  with the horizontal, $b_k$ has angle $(2k+1)\phi$ mod $2\pi$. We put
$A=\{a_k: k\in \mathbb N \} \cup \{(0,0)\}$, $B=\{b_k: k\in \mathbb N \} \cup \{(0,0)\}$, then 
$A \cap B = \{(0,0)\}$ and $P_B(a_k)=b_k$, $P_A(b_k)=a_{k+1}$ by adapting the argument
in example \ref{example3}. The sequence represents again
a discrete version of the logarithmic spiral, turning inwards counterclockwise.  However, if we now choose
$\phi$ such that $\phi/(2\pi)$ is irrational, there will be no periodicity, and the set of directions
$a_k/\|a_k\|$ will be dense in $\mathbb S^1$, and so for $b_k/\|b_k\|$. We have 
$\angle (b^+-a^+,b-a^+)=\pi-\phi$, which means $A,B$ intersect $0$-separably at $(0,0)$.
They intersect {\em at an angle}, this angle being $\pi-\phi$.
However,
$A,B$ do not intersect linearly regularly in the sense of \cite{luke,bauschke1}. Indeed, let us fix
$\psi\in [0,2\pi)$ and $u=(\cos\psi,\sin\psi)$. Then there exist rays $2k\phi$ arbitrarily close to $\psi$
and $a_k$ on these rays, projected from $b_{k-1}$ on ray $(2k-1)\phi$. That means,
$u_k= (b_k-a_k)/\|b_k-a_k\|$ gets arbitrarily close to the direction $u^\perp = (-\sin\psi,\cos\psi)$,
so $u^\perp \in N_A^B(0,0)$. This shows $N_A^{\widetilde{B}}(0,0) = \mathbb R^2$ and
$N_B^{\widetilde{A}}(0,0)=\mathbb R^2$ for $\widetilde{A}=P_A(\partial B\setminus A)$, $\widetilde{B}=P_B(\partial A\setminus B)$, 
so linear regularity and extensions fail.
\hfill $\square$
\end{example}

\begin{example}
{\bf (Spiral and cylinder  \cite{douglas})}. We consider the cylinder
$B=\{(\cos\alpha,\sin\alpha,h): \alpha\in [0,2\pi], 1 \geq h\geq 0\}$ and the spiral 
$A = \{(1+e^{-t})\cos t,(1+e^{-t})\sin t, e^{-t/2}): t \geq 0 \} \cup S$, where
$S=\{(\cos \alpha,\sin\alpha,0): \alpha\in [0,2\pi]\}$. Clearly $A \cap B = S$. As shown 
in \cite{douglas},
any sequence of alternating projections between $A$ and $B$ started at $a\in A\setminus S$
wanders down following the spiral, turning infinitely often around the cylinder with
shrinking $a_n-b_n\to 0$. In particular,
every $x^*\in S$ is an accumulation point of $a_n,b_n$, so convergence fails. Since $B$ is clearly
H\"older regular with respect to $A$, we deduce that the angle condition (\ref{angle})
must fail, so in particular $A$ is not subanalytic. This is interesting, as $A$ is the 
projection of a semianalytic set in $\mathbb R^4$.  For a picture see \cite{douglas}.
\end{example}

\begin{example}
\label{failure}
{\bf (Failure of intrinsic transversality).}
We consider the sets $A=\{2^{-2n}: n\in \mathbb N\}\cup \{0\}$,
$B = \{2^{-2n+1}: n \in \mathbb N\}\cup \{0\}$ in $\mathbb R$, so that $A \cap B=\{0\}$.
The sequence of alternating projections is
$1,\frac{1}{2},\frac{1}{4},\dots$ and converges Q-linearly to 0. We have 
$N_A^B(0)=N_B^A(0)=\mathbb R_+$, hence $A,B$ intersect with the CQ in the sense of
\cite{bauschke1}  at 0,
hence also $0$-separably. Note that $B$ is not $(A,\epsilon,\delta)$-regular
at $0$ in the sense of \cite{bauschke1}, but it is $\sigma$-H\"older-regular for every $\sigma\in [0,1)$. Note that
intrinsic transversality fails here, because it uses the cones $N_A(a)$,
$N_B(b)$, which in this case are too large because they coincide with the whole line.

We modify this example as follows. Let $a_n=2^{-n}$,
$A=\{a_n: n\in \mathbb N\}\cup\{0\}$, $b_n = \frac{1}{2}(a_n+a_{n+1})-\delta_n$, 
$B = \{b_n: n\in \mathbb N\}\cup\{0\}$, where $\delta_n < 2^{-n}(a_n-b_n)$. Then
$\|a_{n+1}-b_n\|$ shrinks only by a factor $1-\delta_n\to 1$ with respect to
$\|b_n-a_n\|$, while shrinkage between $\|a_{n+1}-b_n\|$ and
$\|a_{n+1}-b_{n+1}\|$ is by a factor close to $\frac{1}{2}$. This shows that an
alternating sequence may converge R-linearly without a
fixed shrinkage factor $1-\kappa^2$ in every half step. Note that Theorem \ref{linear}
still applies in this case. 
\end{example}

\begin{example}\label{tangential:nonseparable}
We give an example where $A,B$ intersect tangentially, but not
$\omega$-separably for any $\omega\in [0,2)$. Let $f:\mathbb R\to \mathbb R$ be differentiable with $f'$
continuous at $0$, $f(0)=0$, $f(x)> 0$ for $x\not=0$,
and define
$A={\rm epi}f=\{(x,y)\in \mathbb R^2:  y \geq f(x)\}$, $B=\{(x,y)\in \mathbb R^2: y \leq 0\}$,
then $A \cap B=\{(0,0)\}$. We consider a building block
$b\to a^+\to b^+$. Let $b^+=(x,0)$, then $a^+=(x,f(x))$. Suppose $b=(y,0)$, then  $y = x + f(x)f'(x)$. 
Then the quotient $q$ in $(\ref{angle}')$ reads
\[
q(x)=
\frac{\sqrt{1+f'(x)^2}-1}{f(x)^\omega \sqrt{1+f'(x)^2}}
= \frac{f'(x)^2}{f(x)^\omega \sqrt{1+f'(x)^2} (\sqrt{1+f'(x)^2}+1)}
\leq \frac{f'(x)^2}{2 f(x)^\omega}.
\]
Therefore, if the angle condition (\ref{angle}) is to hold for some $\omega$, then
$\liminf_{x\to 0} \frac{f'(x)^2}{f(x)^{\omega}}\geq \gamma >0.$ It is possible to construct $f$ such that this fails
for every $\omega\in [0,2)$.
Take for instance
$$f(x) =\left\{\begin{array}{ll}
e^{-\frac{1}{x^2}} &\mbox{if } x\neq 0\\
0&\mbox{if } x=0
\end{array}\right.,
$$
then $q(x)  \leq 2 x^{-6} \exp(-x^{-2}(2-\omega)) \to 0$ as $x\to 0$ for $0 < \omega < 2$. Separability with $\omega=0$ is also impossible because $f'(x)\to 0$ as $x\to 0$. 
In conclusion, the sets $A$ and $B$ intersects tangentially,  but not separably for any $\omega\in[0,2)$.
\end{example}

\begin{example}
\label{weak}
Using the same function $f$ and $A,B$, observe that
for $\omega \geq 2$ the quotient $q(x)$ stays away from 0,
so that condition ($\ref{angle}'$) is satisfied.  
This explains why values $\omega \geq 2$ are not meaningful
in definition \ref{separable}.
\end{example}


\begin{thebibliography}{99} 

\bibitem{attouch} H. Attouch, J. Bolte, P. Redont, A. Soubeyran. Proximal alternating minimization 
and projection methods for nonconvex problems: an approach based on the 
Kurdyka-\L ojasiewicz inequality. Mathematics of Operations Research, 35(2):2010, 438--457.

\bibitem{bauschke-survey} H. H. Bauschke,   J. M. Borwein. On projection algorithms for solving
convex feasibility problems. SIAM Review 38:1996, 367 -- 426.

\bibitem{BCL} H.H. Bauschke, P.L. Combettes, D.R. Luke. Phase retrieval, error reduction algorithm, 
and Fienup  variants: a view from convex optimization.
J. Opt. Soc. Am. Series A, 19(7):2002, 1334--1345.

\bibitem{bauschke1} H.H. Bauschke, D.R. Luke, H.M. Phan, X. Wang.
Restricted normal cones and the method of alternating projections:
theory.  Set-Valued and Variational Analysis, 21:2013, 431 -- 473.  

\bibitem{bauschke2} H.H. Bauschke, D.R. Luke, H.M. Phan, X. Wang. Restricted normal cones and the 
method of alternating projections: applications. Set-Valued and Variational Analysis, 21:2013, 475--501.

\bibitem{bauschke-noll} H.H. Bauschke, D. Noll. On cluster points of alternating projections.
Serdica Math. Journal, 39:2013, 355 -- 364.  

\bibitem{douglas} H.H. Bauschke, D. Noll. On the local convergence of the Douglas-Rachford
algorithm. Archiv der Mathematik, 102(6):2014, 589--600. 

\bibitem{bierstone}
E. Bierstone, P. Milman. Semianalytic and subanalytic sets. IHES Publ. Math., 67:1988, 5-42.

\bibitem{bolte1} J. Bolte, A. Daniilidis, A.S. Lewis. The \L ojasiewicz inequality for nonsmooth subanalytic
functions with applications to subgradient dynamical systems. SIAM J. Opt. 17(4):2007, 1205-1223. 

\bibitem{borwein} J.M. Borwein, G. Li, L. Yao. Analysis of the convergence rate for the cyclic projection 
algorithm applied to basic semialgebraic convex sets. 
SIAM Journal on Optimization, 24(1):2014, 498--527.

\bibitem{combettes} P.L. Combettes, H.J.  Trussell. Method of successive projections for finding a common point of sets in metric space. Journal of Optimization Theory and Applications, 67:1990, 487--507.

\bibitem{obsolete} D. Drusvyatskiy, A.D. Ioffe, A.S. Lewis. Alternating projections
and coupling slope. {\tt arXiv:1401.7569v1 [math.OC] 29 Jan 2014}

\bibitem{elser}  V. Elser. Solution of the crystallographic phase problem
by iterated projections. Acta Crystallographica Section A. 59:2003, 201--209.

\bibitem{fabian} M. Fabian. Lipschitz smooth points of convex functions and isomorphic characterizations of Hilbert space.
Proc. London Math. Soc. s3-51:1985, 113--126.

\bibitem{federer} H. Federer. Hausdorff measure and Lebesgue area. Proc. Nat. Acad. Sci.
U.S.A. 37:1951, 90 -- 94.

\bibitem{gerchberg}
R. W. Gerchberg and W. O. Saxton. A practical algorithm for the determination of the phase from image and diffraction plane pictures. Optik 35: 1972, 237--246.

\bibitem{kruger}  A.Y. Kruger. About regularity of collections of sets. Set-Valued Analysis, 14(2):2006, 187--206.

\bibitem{schirotzek} A.Y. Kruger. 
On Fr\'echet subdifferentials. Optimization and related topics, 3,
J. Math. Sci. (N. Y.),
Journal of Mathematical Sciences (New York),
 116(3):2003,  3325--3358.

\bibitem{malick} A.S. Lewis, J. Malick. Alternating projections on manifolds.
Math. Oper. Res. 33:2008, 216--234.

\bibitem{luke} A.S. Lewis, R. Luke, J. Malick. Local linear convergence for alternating and averaged
non convex projections. Found. Comp. Math. 9:2009, 485--513.

\bibitem{mord} B.S. Mordukhovich. Variational Analysis and Generalized Differentiation. Springer, New York, 2006.

\bibitem{initial} D. Noll, A. Rondepierre. On local convergence of the method of alternating projections.
{\tt arXiv:1312.5681v1 [math.OC] 19 Dec 2013}

\bibitem{neumann} J. von Neumann. Functional Operators, vol. II  (Princeton University Press, Princeton, 1950).

\bibitem{thibault} R.A. Poliquin, R.T. Rockafellar, L. Thibault. Local differentiability of distance
functions. Trans. Amer. Math. Soc. 352:2000,  5231--5249.

\bibitem{rock} R.T. Rockafellar, R.J.B. Wets. Variational Analysis. Grundlehren der mathematischen
Wissenschaften, Springer 317:2009.

\bibitem{schwarz} H.A. Schwarz. \"Uber einige Abbildungsaufgaben.
Gesammelte Mathematische Abhandlungen, 11:1869, 65 -- 83.

\bibitem{shiota}  M. Shiota. Geometry of Subanalytic and Semialgebraic Sets. Birkh\"auser Verlag  1997.

\end{thebibliography}
\end{document}